\documentclass[letterpaper, 10pt, journal, final]{IEEEtran}  
\IEEEoverridecommandlockouts 

\usepackage{mathtools} 
\usepackage{graphicx} 
\usepackage{amsmath} 
\usepackage{amssymb}  
\usepackage{hyperref}
\usepackage{multirow,subcaption}
\usepackage{color}
\usepackage{algorithm, algorithmicx, algpseudocode}

\newcommand{\R}{{\mathbb R}}

\newcommand{\Rnn}{{\mathbb R}_{\ge 0}}
\newcommand{\Rp}{{\mathbb R}_{> 0}}
\newcommand{\C}{{\mathbb C}}

\newcommand{\cB}{{\mathcal B}}
\newcommand{\cC}{{\mathcal C}}
\newcommand{\cD}{{\mathcal D}}

\newcommand{\cG}{{\mathcal G}}

\newcommand{\cK}{{\mathcal K}}

\newcommand{\cR}{{\mathcal R}}

\newcommand{\tahn}{{\mathrm{tahn}}}
\newcommand{\diag}{{\mathrm{diag}}}

\newcommand{\bfone}{\mathbf 1}

\newcommand{\inter}{\mathrm{int}}

\newcommand{\pfn}{\rm{PF_n}}
\newcommand{\wpfn}{\rm{WPF_n}}

\def\QED{\mbox{\rule[0pt]{1.3ex}{1.3ex}}} 
\newenvironment{proof}{{\quad \bf Proof:\,}}{\hfill \QED\par}
\newenvironment{proof-of}[1]{{\,\, \bf Proof of #1:\,}}{\hfill\QED\par}
\newtheorem{thm}{Theorem}
\newtheorem{cor}{Corollary}

\newtheorem{defn}{Definition}

\newtheorem{prop}{Proposition}

\newtheorem{exmp}{Example}
\newtheorem{rem}{Remark}

\title{
	Operator-Theoretic Characterization of Eventually Monotone Systems
		\thanks{Aivar Sootla is with the Department of Engineering Science, University of Oxford, Parks Road, Oxford, OX1 3PJ, UK {\tt aivar.sootla@eng.ox.ac.uk}.}
		\thanks{Alexandre Mauroy is with Namur Center for Complex Systems (naXys) and Department of Mathematics, University of Namur, B-5000, Belgium {\tt alexandre.mauroy@unamur.be}}
\thanks{Most of this work was performed when A. Sootla and A. Mauroy were with the University of Li\`{e}ge and held a F.R.S--FNRS fellowship and a return grant from the Belgian Science Policy (BELSPO), respectively.}
	}
\begin{document}
\maketitle 
\IEEEpeerreviewmaketitle
\begin{abstract}
Monotone systems are dynamical systems whose solutions preserve a partial order in the initial condition for all positive times. It stands to reason that some systems may preserve a partial order only after some initial transient. These systems are usually called eventually monotone. While monotone systems have a characterization in terms of their vector fields (i.e. Kamke-M\"uller condition), eventually monotone systems have not been characterized in such an explicit manner. In order to provide a characterization, we drew inspiration from the results for linear systems, where eventually monotone (positive) systems are studied using the spectral properties of the system (i.e. Perron-Frobenius property). In the case of nonlinear systems, this spectral characterization is not straightforward, a fact that explains why the class of eventually monotone systems has received little attention to date. In this paper, we show that a spectral characterization of nonlinear eventually monotone systems can be obtained through the Koopman operator framework. We consider a number of biologically inspired examples to illustrate the potential applicability of eventual monotonicity.
\end{abstract}
\section{Introduction}

The study of dynamical systems whose trajectories preserve a partial order yielded a number of strong stability properties (cf.~\cite{smith2008monotone}). Such systems are called \emph{monotone} in the literature. Besides being an important topic of theoretical research, they have also a great impact in numerous applications such as economics (cf.~\cite{leontief1986input}), biology (cf.~\cite{sontag2007monotone}), and control theory (cf.~\cite{angeli2003monotone}). 

One of the major properties of linear monotone systems (or simply \emph{positive systems}) is that their trajectories remain nonnegative given a nonnegative initial condition, where nonnegativity is understood entry-wise. A number of results for positive systems were derived using the celebrated Perron-Frobenius theorem describing strong spectral properties of nonnegative matrices. These spectral properties are sometimes collectively called \emph{Perron-Frobenius} properties (cf.~\cite{elhashash2008general,noutsos2006perron}). To our best knowledge, it was first noticed by~\cite{friedland1978inverse} that some matrices with a few negative entries possess the Perron-Frobenius properties. It was later observed that there are systems whose trajectories become positive only after an initial transient, given a positive initial condition. The proof of this remarkable fact relies on the Perron-Frobenius property of the matrices, and lead to a number of results in linear algebra in the context of \emph{eventual positivity} (\cite{zaslavsky1999jordan,noutsos2008reachability,olesky2009m}). Recently eventual positivity has also been introduced to control theory (\cite{altafini2015predictable,altafini2015realizations}), while providing a number of powerful results. For example, in~\cite{altafini2015realizations}, a realization result for input-output positive systems was established using eventual positivity. 

Nonlinear asymptotically or \emph{eventually monotone} systems have received little attention to date. They were considered in \cite{hirsch1985systems}, where strong convergence results were derived, and more recently in \cite{wang2008singularly}, but only implicitly. These systems received increasing attention only recently, probably because a spectral characterization of nonlinear systems is not straightforward. Such spectral characterization of a nonlinear system can be obtained through the properties of the so-called Koopman operator (also called composition operator). This operator has been introduced in 1931 by \cite{Koopman1931} (see also \cite{budivsic2012applied} for a review) and since then, its spectrum has been extensively studied in a theoretical context (see e.g. \cite{caughran1975spectra,ding1998point,ridge1973spectrum,shapiro1987essential}) and in the case of dynamical systems (\cite{Gaspard1995}). More recently, a number of studies focused on the eigenfunctions of the Koopman operator, starting with the seminal work by \cite{mezic2005} and investigating the interplay between the eigenfunctions and the geometric properties of the systems (e.g., phase reduction by~\cite{mauroy2013isostables}, stability analysis by~\cite{mauroy2014global}). The present paper adds another result in this framework. Specifically, it provides a spectral characterization of nonlinear eventually monotone systems using the Koopman operator framework. However, the analysis is limited to the basin of attraction of an exponentially stable equilibrium. If the equilibrium is stable, but not exponentially stable, the Koopman operator may have a continuous spectrum (cf.~\cite{gaspard2005liouville}) and the spectral properties of the system are more difficult to ascertain.

In this paper, we first (re-)derive some results for linear eventually positive systems, where we follow the development by~\cite{zaslavsky1999jordan,noutsos2008reachability,olesky2009m}. We then translate these results to the nonlinear setting and provide a characterization of eventually monotone systems, which mirrors the linear case (i.e. eventual positivity). In particular, we show that an eventually monotone system with an exponentially stable equilibrium has a real, negative eigenvalue of the Jacobian matrix at this equilibrium. We also prove that the corresponding eigenfunction of the Koopman operator, which can be seen as an infinite-dimensional eigenvector, is monotone in some sense.

For systems with continuously differentiable vector fields, there is another nonlinear generalization of positivity of linear systems, which is called \emph{differential positivity} and was recently introduced by~\cite{forni2015differentially}. Differential positivity is a more general concept than monotonicity, in the sense that it is defined with a more general partial order related to a cone field instead of a constant cone. We establish the conditions for a differentially positive system to be eventually monotone. This indicates that eventual monotonicity can offer a trade-off between the class of monotone systems and the class of differentially positive systems. We also provide a tool to compute candidate cones with respect to which the system is strongly eventually monotone. To our best knowledge, there exists no equivalent tool for monotone systems. We consider a number of examples inspired by models from biological and biomedical applications. In particular, we show that the gut kinetics subsystem in the glucose consumption model for type I diabetic patients (\cite{man2006system}) and the toxin-antitoxin system (\cite{cataudella2013conditional}) are strongly eventually monotone.

The rest of the paper is organized as follows. In Section~\ref{s:prel}, we briefly introduce monotone systems and provide our main motivation for considering eventual monotonicity. The results related to linear eventually positive systems are covered in Section~\ref{s:even-pos}. They are extended to the nonlinear setting in Section~\ref{s:even-monot}, where we also discuss the relationship between eventual monotonicity and differential positivity and derive a tool to compute candidate cones with respect to which the system is eventually monotone. 
Our theoretical results are illustrated with examples in Section~\ref{sec:examples}. Concluding remarks are given in Section~\ref{s:con} and most of the proofs are in Appendix~\ref{app:proofs}.

\section{Preliminaries}\label{s:prel}

\subsection{Notations and standing assumptions}
We consider dynamical systems
\begin{equation}
\label{eq:sys}
\dot x = f(x),
\end{equation} 
with $f: \cD \rightarrow \R^n$ and where $\cD$ is an open subset of $\R^n$. We will assume that the vector field is continuously differentiable on $\cD$ ($C^2$ for some results), which ensures existence, uniqueness, and continuity of solutions. We define the flow map $\phi: \R \times \cD \rightarrow \R^n$, where $\phi(t, x_0)$ is a solution to the system~\eqref{eq:sys} with an initial condition $x(0) = x_0$. Throughout the paper, we make the standing assumption that the system admits an equilibrium $x^\ast$. When the equilibrium is attracting, we denote its basin of attraction by $\cB(x^\ast) = \{ x \in \cD \bigl| \lim_{t\rightarrow \infty} \phi(t, x) = x^\ast \}$. We also denote the Jacobian matrix $(\partial f/\partial x)$ by $J(x)$. The eigenvalues of $J(x^\ast)$ are $\{\lambda_1, \dots, \lambda_n\}$ (counted with their algebraic multiplicities). We order them according to their real parts, i.e. $\Re(\lambda_i) \ge \Re(\lambda_j)$ for all $i \le j$. The corresponding right and left eigenvectors of $J(x^\ast)$ are denoted by $v_i$ and $w_i$, respectively. The spectral radius $\rho(A)$ of a matrix $A\in\R^{n\times n}$ is equal to $\max|\lambda_i|$, where $\lambda_i$ are the eigenvalues of $A$. The matrix $A^T$ denotes the transpose of the matrix $A$.

We list below for future reference a few standing assumptions on (the equilibrium of) the system~\eqref{eq:sys}.
\begin{itemize}
	\item[A1.] The equilibrium point is asymptotically stable; 
	\item[A2.] The eigenvalues of the Jacobian matrix $J(x^\ast)$ satisfy $\Re(\lambda_i) \neq 0$ for all $i$;
	\item[A3.] The Jacobian matrix $J(x^\ast)$ is diagonalizable, i.e. the eigenvectors $v_i$ are linearly independent and the algebraic multiplicity of an eigenvalue is equal to its geometric multiplicity.
\end{itemize}
Note that when Assumptions A1 and A2 are satisfied together, $J(x^\ast)$ is Hurwitz and the equilibrium is exponentially stable.

\subsection{Partial order and monotonicity}

We will study the properties of the system \eqref{eq:sys} with respect to a partial order. A relation $\succeq$ is called a {\it partial order} if it is reflexive ($x\succeq x$), transitive ($x\succeq y$, $y\succeq z$ implies $x\succeq z$), and antisymmetric ($x\succeq y$, $y\succeq x$ implies $x = y$). Partial orders can be defined with cones $\cK\subset\R^n$. A set $\cK$ is a \emph{positive cone} if $\Rnn \cK \subseteq \cK$, $\cK + \cK \subseteq \cK$, $\cK\cap \cK \subseteq \{ 0\}$. We will use the following cone $\cK_L$, which is called \emph{Lorentz} and is characterized as follows:
\begin{gather*}
\cK_L = \left\{ x \in \R^n \Bigl| \sum\limits_{i=2}^n x_i^2 \le x_1^2,\, x_1 \ge 0 \right\}.
\end{gather*}
We define a partial order as: $x\succeq_\cK y$ if and only if $ x - y \in \cK$. We write $x\not \succeq_\cK y$ if the relation  $x \succeq_\cK y$ does not hold. We will also write $x\succ_\cK y$ if $x\succeq_\cK y$ and $x\ne y$, and $x\gg_\cK y$ if $x- y \in \inter(\cK)$. From this point on, we use the notations $\succ$, $\succeq$, $\gg$ if $\cK$ is the nonnegative orthant $\Rnn^n = \{x\in\R^n | x_i \ge 0\, \forall i =1, \dots,n \}$. We say that the function $g:\R^n \rightarrow \R$ is monotone with respect to the cone $\cK$ if for all $x\succeq_\cK y$, we have $g(x) \ge g(y)$. A set $M$ is called \emph{p-convex} if, for every $x,y \in M$ such that $x\succeq_\cK y$ and every $\lambda\in(0,1)$, we have $\lambda x+ (1-\lambda) y\in M$. We also use order-intervals $[x, y]_\cK =\left\{z \in \R^n \Bigl| x \preceq_\cK z \preceq_\cK y \right\} \subset \R^n$, with $x,y\in \R^n$. Systems whose flows preserve a partial order relation $\succeq_\cK$ are called \emph{monotone systems}. 
\begin{defn}\label{def:mon}
The system is \emph{monotone} on $\cD$ with respect to the cone $\cK$ if $\phi(t,x)\preceq_\cK \phi(t,y)$ for all $t\ge 0$ and for all $x$,$y\in\cD$ such that $x\preceq_\cK y$. The system is called \emph{strongly monotone} on $\cD$ with respect to the cone $\cK$ if it is monotone and if $\phi(t,x) \ll_\cK \phi(t,y)$ holds for all $t>0$ provided that $x$,$y\in\cD$ and $x\prec_\cK y$.
\end{defn}
There exists a certificate for monotonicity with respect to an orthant, which is called \emph{Kamke-M\"uller} conditions (cf.~\cite{smith2008monotone}). For differentiable vector fields, these conditions imply the following result.
\begin{prop}\label{prop:kamke}
Consider the system~$\dot x = f(x)$, where $f$ is differentiable, and let the set $\cD$ be p-convex. Then
\begin{gather*}
\frac{\partial f_i}{\partial x_j}\ge 0,\quad\forall~i\ne j,\quad \forall x\in\cD
\end{gather*}
if and only if the system~\eqref{eq:sys} is monotone on $\cD$ with respect to $\Rnn^n$.
\end{prop}
This result may be extended to other orthants and cones in $\R^n$ (cf.~\cite{walcher2001cooperative}). In order to continue our discussion we introduce the following notions.

\begin{defn}
A matrix $A \in\R^{n\times n}$ is a \emph{Metzler matrix} if its off-diagonal elements are nonnegative. A matrix $A \in\R^{n\times n}$ is \emph{reducible} if there exist a permutation $T$ and an integer $k$ such that 
\begin{gather*}
T A T^{-1} = \begin{pmatrix}
B & C \\
0 & D
\end{pmatrix},
\end{gather*}
where $B \in \R^{k\times k}$, $D\in \R^{n-k\times n-k}$, $C\in\R^{k\times n-k}$ and $0$ stands for the $(n-k) \times k$ zero matrix. If no such $T$ and $k$ exist, then the matrix is called \emph{irreducible}.
\end{defn}

In light of this definition, we state that a system~\eqref{eq:sys} is monotone if and only if the Jacobian  matrix $J(x)$ of the vector field is a Metzler matrix for all $x\in \cD$. If $J(x)$ is an irreducible Metzler matrix for all $x\in\cD$, then the system is strongly monotone (cf.~\cite{smith2008monotone}). Additionally, if  the system is linear (in which case the Jacobian matrix is constant), then monotonicity is equivalent to \emph{positivity}, which means that the system is invariant with respect to $\Rnn^{n}$. 

\subsection{From monotonicity to eventual monotonicity}
Monotonicity is a strong property which leads to a number of strong results in dynamical systems. However, some studies (eg.~\cite{sontag2007monotone,wang2008singularly}) suggest that many systems exhibit \emph{a near monotone} behavior, a property which can potentially lead to results similar to the case of monotone systems. This is illustrated with the following two examples.
\begin{exmp}\label{ex:linear}
Consider the linear system with $0<\epsilon \ll 1$:
\begin{gather*}
\begin{pmatrix}
\dot x_1 \\  \dot x_2 \\ \varepsilon \dot x_3
\end{pmatrix} = \underbrace{\begin{pmatrix}
     -6  &  10&  4\\
     -7  &  2 & 12 \\
      3  & -3 & -4
\end{pmatrix}}_{A} \begin{pmatrix}
x_1 \\ x_2 \\ x_3
\end{pmatrix}. 
\end{gather*}
 We can eliminate the variable $x_3$ using singular perturbation theory and obtain the system
\begin{align*}
&\begin{pmatrix}
\dot {\tilde x}_1 \\ \dot {\tilde x}_2
\end{pmatrix}  = \tilde A \begin{pmatrix}
\tilde x_1 \\ \tilde x_2
\end{pmatrix}
\end{align*}
where
\begin{gather*}
\tilde A  = \begin{pmatrix}
     -6  &  10 \\
     -7  &  2 \end{pmatrix} 
		+ \frac{1}{4} \begin{pmatrix}  4 \\ 12
\end{pmatrix} 
\begin{pmatrix}
      3 \\ -3
\end{pmatrix}^T = \begin{pmatrix}
   -3  &  7\\
    2  & -7
\end{pmatrix}.
\end{gather*}

The system matrix $\tilde A$ of the reduced order system is Metzler, while the matrix $A$ is not. It can be verified that the states $x_1(t)$, $x_2(t)$ converge to $\tilde x_1(t)$, $\tilde x_2(t)$, while $x_3(t)$ converges to a constant. Hence, after some time $\tau_0>0$, the flow of the full system $x(t)$ becomes positive for all initial conditions $x(0) \in \Rnn^3$.
\end{exmp}
\begin{exmp}\label{ex:sing-pert}
For nonlinear systems, the argument is similar. Consider the system with $0<\varepsilon \ll 1$:
  \begin{equation}
    \label{sys-full}
    \begin{aligned}
      \dot x_1 &= \dfrac{10}{1 + x_2^2} - x_1, \\
      \dot x_2 &= \underbrace{\dfrac{x_1}{x_1+1}}_{h(x_1)} + 3 x_3  - x_2, \\
      \varepsilon\dot x_3 &= \dfrac{1}{1 + x_1} - x_3. 
    \end{aligned}
  \end{equation}
It can be verified that the system~\eqref{sys-full} is not monotone with respect to any orthant, which is due to the function $h(x_1)$ in the second equation. However, for $\varepsilon\rightarrow 0$, the system~\eqref{sys-full} is reduced to
    \begin{equation}
    \label{sys-red}
    \begin{aligned}
      \dot{\tilde x}_1 &= \dfrac{10}{1 + \tilde x_2^2} - \tilde x_1, \\
      \dot{\tilde x}_2 &= 1 + \dfrac{2}{\tilde x_1+1}- \tilde x_2, 
    \end{aligned}
  \end{equation}
which is a monotone system with respect to the orthant $\diag\{1, -1\}\Rnn^2$. After some time $\tau_0>0$, the flow of the full system~\eqref{sys-full} converges to the flow of the reduced system~\eqref{sys-red}. Hence the flow of~\eqref{sys-full} may still preserve the partial order in all three variables. 
\end{exmp}
These examples suggest that it is desirable to consider a class of systems larger than monotone systems: the class of systems that are asymptotically monotone. These systems are not limited to near monotone systems in the context of singular perturbation theory. We call them \emph{eventually monotone} systems.

\section{Eventually Positive Linear Systems}
\label{s:even-pos}

\subsection{(Strongly) Eventually Positive Dynamical Systems}
In this section, we present a few results on linear eventually positive systems. These results will be extended to nonlinear eventually monotone systems in Section \ref{s:even-monot}.
\begin{defn}\label{def:ev-pos}
The system $\dot x = A x$ is \emph{eventually positive} if for any $x\succeq 0$ there exists $\tau_0\ge 0$ such that $\phi(t, x) \succeq 0$ for all $t\ge \tau_0$. The system is \emph{strongly eventually positive} if it is eventually positive and for any $x\succ 0$ there exists $\tau_0\ge 0$ such that $\phi(t,x)\gg 0$ for all $t\ge \tau_0$. 
\end{defn}

(Strong) eventual positivity is a generalization of (strong) positivity, which is achieved by allowing $\tau_0$ be larger than zero. While positivity of a system can be characterized by the sign pattern of the matrix $A$ (Kamke-M\"uller condition), eventual positivity of a system has a characterization through spectral properties of the matrix $A$. These spectral properties stem from the celebrated Perron-Frobenius theorem, which states that for any irreducible nonnegative matrix $A$ its spectral radius $\rho(A)$ is a simple (positive) eigenvalue and the corresponding eigenvectors can be chosen to be positive. This leads to the following definition (cf.~\cite{elhashash2008general}, \cite{zaslavsky1999jordan}). A matrix $A \in \R^{n\times n}$ is said to possess the {\it (weak) Perron-Frobenius property} if $\rho(A)$ is a positive eigenvalue of $A$ and the corresponding right eigenvector can be chosen to be nonnegative. A matrix $A\in \R^{n\times n}$ possesses the {\it strong Perron-Frobenius property} if $\rho(A)$ is a simple, positive eigenvalue of $A$ and the corresponding right eigenvector can be chosen to be positive. We denote by $\wpfn$ (respectively, $\pfn$) the class of matrices $A\in \R^{n\times n}$ such that $A$ and $A^T$ possess the Perron-Frobenius property (respectively, the strong Perron-Frobenius property). 

Some matrices can have a small negative entry, while the rest of the entries are nonnegative, and still possess the Perron-Frobenius property. For instance, we can take the matrix $B = e^{A t}$ with a small enough $t$ and $A$ from Example~\ref{ex:linear}. In fact, it can be verified that the matrix $B^{k}$ becomes positive for a large enough $k$. This leads to the following definition. A matrix $A \in \R^{n\times n}$ is called eventually nonnegative (respectively, eventually positive) if there exists $k_0 \in\mathbb{N}$ such that, for all $k\ge k_0$, the matrices $A^k$ are nonnegative (respectively, positive). The relationship between Perron-Frobenius properties and eventually positive/nonnegative matrices is characterized by the following inclusions of the sets of matrices:
\begin{multline} \label{inclusion:pfn}
\pfn = \{\textrm{Eventually positive matrices}\}  \\
\subset \{\textrm{Nonnilpotent, eventually nonnegative matrices} \} \\
\subset \wpfn
\end{multline}
where $A$ is nilpotent if there exists an integer $k \in \mathbb{N}$ such that $A^k = 0$. Every inclusion is shown to be strict by finding a suitable counterexample (cf.~\cite{elhashash2008general}).

Now we characterize eventually positive systems $\dot x = Ax$, where we follow the developments by~\cite{noutsos2008reachability,olesky2009m}, but we restrict ourselves by additional assumptions. 

\begin{prop}\label{prop:ev-pos-dyn}
Consider the system $\dot x = A x$ satisfying Assumption A3, with $\lambda_j$ being the eigenvalues of $A$. \\
(i) If the system is eventually positive, then $\lambda_1$ is real, and the right and left eigenvectors $v_1$, $w_1$ of $A$ corresponding to $\lambda_1$ can be chosen to be nonnegative;\\
(ii) Furthermore, the system is strongly eventually positive if and only if $\lambda_1$ is simple, real, $\lambda_1 > \Re(\lambda_j)$ for all $j \ge 2$, and the right and left eigenvectors $v_1$, $w_1$ of $A$ corresponding to $\lambda_1$ can be chosen to be positive.
\end{prop}
 The proof is found in~\cite{sootla2015evpos} and in Appendix~\ref{app:proofs}. Proposition~\ref{prop:ev-pos-dyn} establishes that strong eventual positivity of a system is a spectral condition on the matrix $A$, which is necessary and sufficient. Eventual positivity of a system lacks sufficiency, which is consistent with the inclusions in~\eqref{inclusion:pfn}.
We will refer to $\lambda_1$ from Proposition~\ref{prop:ev-pos-dyn} as the \emph{dominant eigenvalue}. Based on the proof of this result, eventual positivity of $\dot x = A x$ rules out a complex eigenvalue $\lambda_j$ for some $j$ such that $\lambda_1 = \Re(\lambda_j)$. We illustrate this property on an example. 
\begin{exmp}
	Consider the matrix 
	\begin{gather*}
	A = \begin{pmatrix}
	-10 & - 10& 14 \\
	4   & 1 & -11 \\
	0 & 3 & -9
	\end{pmatrix}
	\end{gather*}
	whose eigenvalues are $\lambda_1 = - 6$, $\lambda_{2} = -6 + 6 i$, $\lambda_3 = -6 - 6i$. It has positive eigenvectors corresponding to the simple eigenvalue $\lambda_1 = -6$, but the system $\dot x = A x$ is not eventually positive. Indeed, let $v_i$, $w_i$ be the right and left eigenvectors corresponding to $\lambda_i$, then we have:
	\begin{multline*}
	e^{A t} = e^{-6 t} (v_1 w_1^T + e^{6 i t} v_2 w_2^T + e^{-6 i t} v_3 w_3^T) = \\e^{-6 t} (v_1 w_1^T + 2 \Re(e^{6 i t} v_2 w_2^T)).
	\end{multline*}
	Since the magnitude of some entries in $2 \Re(v_2 w_2^T e^{6 i t})$ is larger than the corresponding entries of $v_1 w_1^T$, there exists no $\tau_0$ such that the matrix $e^{A t}$ is positive for all $t>\tau_0$ and the system $\dot x = A x$ is not eventually positive.
	\end{exmp}

\subsection{(Eventual) Positivity with Respect to a Cone }\label{ss:cone-ev-pos}
A system $\dot x = A x$ is said to be \emph{positive with respect to a cone $\cK$} or \emph{$\cK$-positive} if $e^{A t} \cK \subseteq \cK$ for any $t\ge 0$. $\cK$-positive systems are also well-studied in the literature. For example, any matrix $A$ with a simple and real dominant eigenvalue (i.e., $\lambda_1 > \Re(\lambda_j)$ for all $j\ge 2$) leads to a system $\dot x = A x$ that is positive with respect to some cone $\cK$~(\cite{stern1991exponential}). This result raises a question on a relationship between eventually positive and $\cK$-positive systems, which we will investigate in this section. Provided that the system $\dot x = Ax$ satisfies Assumption~A3, consider the following set:
\begin{multline}\label{ice-cream-cones}
\cK_{\alpha}=  \Bigg\{ y \in \R^n \Bigl| \left(\sum\limits_{i =2}^n \alpha_i |w_i^T y|^2\right)^{1/2} \le w_1^T y \Bigg\}\,,
\end{multline}
where $w_i$ are the left eigenvectors of the matrix $A$, and $\alpha_i > 0$ are chosen a priori. Every set $\cK_\alpha$ is a cone since it can be transformed to a Lorentz cone using the transformation $W = \begin{pmatrix} w_1 & \sqrt{\alpha_2} w_2 & \cdots & \sqrt{\alpha_n}w_n\end{pmatrix}^T$ and the change of variables $x = W y$. According to Assumption~A3, the vectors $w_i$ are linearly independent, $W$ is invertible, and our cones are well-defined. We proceed by reformulating a certificate for strong eventual positivity in terms of the cones $\cK_\alpha$ (the proof can be found in~\cite{sootla2015evpos} and in Appendix~\ref{app:proofs}):

\begin{prop}\label{prop:pos-cone}
	Consider the system $\dot x = A x$ satisfying Assumption A3, with $\lambda_j$ being the eigenvalues of $A$. Let $\lambda_1$ be simple, real and negative, and $\lambda_1 > \Re(\lambda_j)$ for all $j\ge 2$. Then: \\
	(i) the system is $\cK_\alpha$-positive for any positive vector $\alpha$;\\
	(ii) the system is strongly eventually positive if and only if there exist positive scalars $\beta_i$ and $\gamma_i$ for $i= 2,\dots,n$ such that $\cK_{\beta} \subset \inter(\Rnn^n)\cup\{0\}$ and  $\Rnn^n\subset \inter(\cK_{\gamma})\cup\{0\}$.  
\end{prop}
It is straightforward to show that $\lim_{t \rightarrow \infty} e^{At} \cK_\alpha= \cK_\infty$, where $\cK_\infty = \{y \in \R^n | y = \Delta v_1, \Delta \in \Rnn \}$ and $v_1$ is the right eigenvector of $A$ associated with the dominant eigenvalue $\lambda_1$. Therefore, the set $\cK_\infty$ acts as an attractor for all the trajectories starting from the set $\{x \in \R^n \bigl| w_1^T x > 0\}$. The trajectories starting from the set $\{x \in \R^n \bigl| w_1^T x < 0\}$ are attracted by $-\cK_\infty$. We conclude this subsection by the following corollary from Proposition~\ref{prop:pos-cone} (the proof can be found in~\cite{sootla2015evpos} and in Appendix~\ref{app:proofs}):
\begin{cor} \label{prop:dom-eig}
	Consider the system $\dot x = A x$ satisfying Assumption~A3, with $\lambda_j$ being the eigenvalues of $A$. Let $\lambda_1$ be simple, real, negative and $\lambda_1 > \Re(\lambda_j)$ for all $j \ge 2$. Then there exists an invertible matrix $S$ such that the system $\dot z = S^{-1} A S z$ is eventually positive. 
\end{cor}
Combining the results of this section gives an unexpected result: under the premise of Proposition~\ref{prop:pos-cone}, any strongly eventually positive system is also $\cK_\alpha$-positive. This seems to undermine the value of eventual positivity as an extension of positivity. However, strong eventual positivity with respect to orthants gives additional properties, which $\cK_\alpha$-positive systems may not necessarily possess (\cite{sootla2015evpos}). Furthermore, there are interesting extensions of these results in the nonlinear case. For sufficiently smooth nonlinear systems, positivity translates to differential positivity, which is a very general property that holds, for instance, on the basin of attraction of an exponentially stable equilibrium (cf.~\cite{forni2015differentially,mauroy2015operator}). In contrast, cone monotonicity is a strong property, which is also hard to check since one needs to find first a candidate cone inducing a partial order. On the other hand, eventual monotonicity offers a trade-off between the concepts of differential positivity and monotonicity. We elaborate on this trade-off in the sequel.
	
\section{Eventually Monotone Systems }\label{s:even-monot}
\subsection{Definitions and Basic Properties}
The definition of eventual monotonicity comes as a natural extension of monotonicity (cf.~\cite{hirsch1985systems}). 
\begin{defn}\label{def:ev-mon}
The system $\dot x = f(x)$ is \emph{eventually monotone} on $\cC\subseteq \cD$ with respect to the cone $\cK$ if for any $x$, $y\in\cC$ such that $x\succeq_\cK y$, there exists $\tau_0\ge 0$ such that $\phi(t, x) \succeq_\cK \phi(t, y)$ for all $t\ge \tau_0$. The system is \emph{strongly eventually monotone} on $\cC$ w.r.t. $\cK$ if it is eventually monotone on $\cC$ w.r.t. $\cK$ and for any $x$, $y\in\cC$ such that $x\succ_\cK y$, there exists $\tau_0\ge 0$ such that $\phi(t,x)\gg_\cK \phi(t,y)$ for all $t\ge \tau_0$. We call a system \emph{uniformly} (strongly) eventually monotone on $\cC$ w.r.t. $\cK$ if it is (strongly) eventually monotone on $\cC$ w.r.t. $\cK$ and $\tau_0$ can be chosen uniformly w.r.t. $\cC$.
\end{defn}

The class of eventually monotone systems is larger than the class of monotone systems, but includes the latter. In particular, every monotone system is eventually monotone with $\tau_0=0$. We note that an arguably more established concept of \emph{eventual strong monotonicity} (cf.~\cite{hirsch2005monotone}) is not equivalent to our definition of \emph{strong eventual monotonicity}. Indeed, eventually strongly monotone systems are monotone, while the strong relation holds after some initial transient. In contrast, we do not require monotonicity for all $t\ge 0$.

If the system is eventually monotone with respect to $\cK$ on $\cC$, then by continuity of solutions we have that $\phi(t, x) \succ_\cK \phi(t,y)$ (respectively, $\phi(t, x) \gg_\cK \phi(t,y)$) for all $t\ge \tau_0$ and for any $x$, $y\in\cC$ such that $x\succ_\cK y$ (respectively, $x\gg_\cK y$). We proceed by presenting some other direct implications of this definition.

\begin{prop}\label{prop:trivial}
Consider that the system~\eqref{eq:sys} with $f\in C^1(\cD)$ has an equilibrium $x^\ast$ and satisfies Assumptions A2--A3. \\ 
(i) If the system~\eqref{eq:sys} is eventually monotone in a neighborhood of $x^\ast$ with respect to a cone $\cK$, then $\lambda_1$ is real, and the right and left eigenvectors $v_1$, $w_1$ of $J(x^\ast)$ corresponding to $\lambda_1$ can be chosen such that $v_1, w_1 \succ_\cK 0$. \\
(ii) If the system~\eqref{eq:sys} is strongly eventually monotone in a neighborhood of $x^\ast$, then $\lambda_1$ is also simple, while $v_1$, $w_1$ can be chosen such that $v_1, w_1\gg_\cK 0$. Conversely, if $\lambda_1$ is simple, real, and $\lambda_1 > \Re(\lambda_j)$ for all $j\geq 2$, then the system is locally strongly eventually monotone in a neighborhood of $x^\ast$ with respect to some cone $\cK$.
\end{prop}
The proof is straightforward by invoking the Hartman-Grobman theorem (\cite{grobman1959homeomorphism,hartman1960lemma}) and Propositions~\ref{prop:ev-pos-dyn} and~\ref{prop:dom-eig}. Additionally, some properties of asymptotically stable monotone systems are preserved in the eventually monotone case.

\begin{prop} \label{prop:order_traject}
Consider that the system~\eqref{eq:sys} with $f\in C^1(\cD)$ is eventually monotone on a forward-invariant open set $\cC$ with respect to the cone $\cK$. \\
(i) If $f(x)\succeq_\cK 0$ for some $x\in\cC$, then there exists $\tau_0 \ge 0$ such that $f(\phi(t,x)) \succeq_\cK 0$ for all $t\ge \tau_0$.\\
If the system admits an asymptotically stable equilibrium $x^\ast$ with a basin of attraction $\cB(x^\ast) \subseteq \cC$, then:\\
(ii) for all $x \in \cB(x^\ast)$ such that $x \succ_\cK x^*$ and for all $t>0$, we have that $f(x) \not \succ_\cK 0$ and $\phi(t,x) \nsucc_\cK x$;\\
(iii) for all $w$, $z\in\cB(x^\ast)$ such that $z\preceq_\cK w$, the order-interval $[z, w]_\cK$ is a subset of $\cB(x^\ast)$.
\end{prop}
The point (i) was shown by~\cite{hirsch1985systems} in a slightly different formulation, while the point (ii) follows from (i) and was also shown by~\cite[Proposition 3.10.]{ruffer2010connection} for the case of monotone systems with vector fields that are not necessarily differentiable. The proof of (iii) is adapted in a straightforward manner from a similar result for monotone systems (see e.g.~\cite{sootla2015pulsesaut}) and can be found in Appendix~\ref{app:proofs}. Convergence results for eventually monotone systems with differentiable flows can also be found in~\cite{hirsch1985systems}. We suspect that many asymptotic properties of monotone systems can be extended to the case of eventually monotone systems. However, some properties of monotone systems, which hold for all $t>0$, may not necessarily hold for eventually monotone systems. For example, according to the point (i) in Proposition~\ref{prop:order_traject}, if a flow is monotone around $x$, then it is monotone for all $t\ge \tau_0$, while this property holds for all $t>0$ in the case of monotone systems. Naturally, a major issue is a certificate for eventual monotonicity, which is provided in the following subsection.

\subsection{Characterization of Eventually Monotone Systems}
\subsubsection{Koopman operator}

In this paper, we propose a characterization of eventually monotone systems, which is based on the framework of the Koopman operator. While many studies have investigated the properties of the operator spectrum in a general context (see e.g. \cite{caughran1975spectra,ding1998point,Gaspard1995,ridge1973spectrum,shapiro1987essential}), we rely here on the properties of the eigenfunctions, and in particular on the interplay between the eigenfunctions and the geometric properties of the system (see e.g.~\cite{mezic2005}).

Consider a space $\mathcal{G}$ (we take $\cG = C^1$, but other choices are possible, see~\cite{mezic2005}) of functions called \emph{observables} $g:\R^n \rightarrow \C$ and suppose that the dynamical system~\eqref{eq:sys} is described by its flow $\phi$. The semi-group of Koopman operators $U^t:\mathcal{F} \to \mathcal{F}$ associated with~\eqref{eq:sys} is defined by
\begin{gather}
U^t g(\cdot) = g \circ \phi(t, \cdot),
\end{gather}
where $\circ$ denotes the composition of functions. Note that one can similarly define the semi-group of operators for vector-valued observables $g:\R^n\rightarrow\C^m$, with some positive integer $m$. The infinitesimal generator of the semi-group is defined by
\begin{equation*}
L = \lim_{t \rightarrow 0} \frac{U^t-I}{t}
\end{equation*}
where $I$ is the identity operator. If the vector field $f$ of~\eqref{eq:sys} and the observables $g:\cC \to \mathbb{C}$ are continuously differentiable on an open set containing a compact set $\cC$, then the infinitesimal generator of the Koopman semi-group is given on $\cC$ by $L g = f^T \nabla g$ (see \cite{Lasota_book}). 

The Koopman operator (i.e. both the semi-group and its generator) is linear (cf.~\cite{mezic2013analysis}), so that it is natural to consider its spectral properties. We define an eigenfunction of the Koopman operator as a function $s \in \mathcal{F}$ that is nonzero and satisfies
\begin{gather}
\label{eq:prop_eigenf}
U^t s(\cdot) = s(\phi(t, \cdot)) = s(\cdot) \, e^{\lambda t}
\end{gather}
where $\lambda\in \mathbb{C}$ is the associated eigenvalue. Equivalently, if the infinitesimal generator is well-defined, we have also
\begin{gather}\label{eq:def_eigenf}
f^T \nabla s  = \lambda s.
\end{gather}
In the linear case $f(x) = A x$, the eigenvalues of $A$ are also eigenvalues of the Koopman operator. Furthermore, if the system satisfies Assumption~A3, then the eigenfunctions corresponding to $\lambda_i$ are given by $s_j(x) = w_j^T x$, where $w_j$ are left eigenvectors of $J(x^\ast)$ (\cite{mezic2013analysis}). Now, consider a nonlinear system~\eqref{eq:sys} that satisfies Assumptions A1--A3 and is characterized by a $C^2$ vector field. Similarly to the linear case, the eigenvalues $\lambda_j$ of $J(x^\ast)$ are the eigenvalues of the Koopman operator and their associated eigenfunctions $s_{j}$ are continuously differentiable in the basin of attraction of $x^\ast$ (\cite{mauroy2014global}). Under these conditions, we can use the eigenfunctions to obtain an expansion of the flow. This is summarized in the following result. The proof can be found in Appendix~\ref{app:proofs}.
\begin{prop}\label{prop:Koopman_expansion}
Consider that the system~\eqref{eq:sys} with $f\in C^2(\cD)$ has an equilibrium $x^\ast$ and satisfies Assumptions A1--A3. Then, for all $x$ in the basin of attraction $\cB(x^\ast)$, the flow can be expressed as
	\begin{equation} \label{eq:expansion}
	\phi(t,x) = x^\ast + \sum \limits_{j=1}^n v_j s_j(x) e^{\lambda_j t}  + R(t,x)
	\end{equation}
	with 
		\begin{equation*}
	R(t,x) = o \left(\left\|\sum \limits_{j=1}^n v_j s_j(x) e^{\lambda_j t} \right\|\right),
	\end{equation*}
	where $v_j$ are the right eigenvectors of $J(x^\ast)$ and $ v_j^T \nabla s_j(x^*) = 1$.
\end{prop}
We note that in the linear case, where  $s_j(x) = w_j^T x$, we recover the classic linear expansion of the flow, with $R(t,x)=0$. If $\lambda_1$ is a simple, real eigenvalue and $\lambda_1> \Re\{\lambda_2\}$, then~\eqref{eq:expansion} implies that the eigenfunction $s_1$ associated with the eigenvalue $\lambda_1$ captures the dominant (i.e. asymptotic) behavior of the system. As shown in the next section, this property is key to our characterization of eventual monotonicity. Moreover, it also follows from \eqref{eq:expansion} that the dominant eigenfunction $s_1$ can be computed by using the Laplace average   
\begin{gather}
\label{eq:Laplace_av}
g_\lambda^{av}(x) = \lim\limits_{t\rightarrow \infty}\frac{1}{t}\int\limits_0^t (g\circ \phi(s, x)) e^{-\lambda s} d s,
\end{gather}
for any $g\in C^1$ that satisfies $g(x^\ast)=0$ and $v_1^T \nabla g(x^\ast) \neq 0$. The Laplace average $g_{\lambda_1}^{av}(x)$ is a projection of $g$ onto $s_1$, which is therefore equal to $s_1$ up to a multiplication with a scalar (see e.g. \cite{mauroy2013isostables}). Note that the Laplace averages diverge if $x$ does not belong to $\cB(x^\ast)$. 
\subsubsection{Main results}
In this section, we present our main result, which shows the relationship between eventual monotonicity and spectral properties of the Koopman operator. We exploit in particular the fact that the dominant eigenfunction $s_1$ captures the asymptotic behavior of the system.

\begin{thm} \label{prop:ev-mon} 
Consider that the system~\eqref{eq:sys} with $f\in C^2(\cD)$ has an equilibrium $x^\ast$ and satisfies Assumptions A1--A3, while $\lambda_j$ are the eigenvalues of $J(x^\ast)$.\\
(i)  If the system is eventually monotone with respect to $\cK$ on a set $\cC\subseteq \cB(x^\ast)$, then $\lambda_1$ is real and negative, the right eigenvector $v_1$ of $J(x^\ast)$ can be chosen such that $v_1 \succ_\cK 0$, while the eigenfunction $s_1$ can be chosen such that $s_1(x) \ge s_1(y)$ for all $x$, $y\in \cC$ satisfying $x\succeq_\cK y$. Furthermore, $s_1(x) > s_1(y)$ for all $x$, $y\in \cC$ satisfying $x \gg_\cK y$. \\
(ii) Furthermore, the system is strongly eventually monotone with respect to $\cK$ on a set $\cC \subseteq \cB(x^\ast)$ if and only if $\lambda_1$ is simple, real and negative, $\lambda_1 > \Re(\lambda_j)$ for all $j \ge 2$, $v_1$ and $s_1$ can be chosen such that $v_1\gg_\cK 0$ and $s_1(x)> s_1(y)$ for all $x$, $y\in \cC$ satisfying $x\succ_\cK y$; \\
(iii) If $\cC$ is compact, then (strong) eventual monotonicity in (i) and (ii) is understood in the uniform sense.
\end{thm}

\begin{proof}
First, we note that $s_1\in C^1(\cB(x^\ast))$, since $\Re(\lambda_i) < 0$ for all $i$ and $f\in C^2(\cD)$ (\cite{mauroy2014global}).\\
(i) By Proposition~\ref{prop:trivial}, $v_1\succ_\cK 0$ and $\lambda_1$ is real, which implies that $s_1$ is a real-valued function. Now, for some monotone observable $g\in C^1$ (satisfying $g(x^\ast)=0$ and $v_1^T \nabla g(x^\ast) \neq 0$), it follows from \eqref{eq:Laplace_av} that
\begin{multline}
\label{eq:increasing_s1}
s_1(x) - s_1(y) = \\\lim\limits_{t\rightarrow \infty} \frac{1}{t} \int\limits_0^t(g\circ\phi(s,x) - g\circ\phi(s,y)) e^{-\lambda_1 s} d s = \\ 
 \lim\limits_{t\rightarrow \infty} \frac{1}{t} \int\limits_{\tau_0}^t(g\circ\phi(s,x) - g\circ\phi(s,y)) e^{-\lambda_1 s} d s
\ge 0
\end{multline}
for all $x\succeq_\cK y$ since $\phi(s,x)\succeq_\cK \phi(s,y)$ and $s \ge \tau_0$.\\
We will prove the second part of the statement by contradiction. Assume that there exist $x\gg_\cK y$ such that $s_1(x) = s_1(y)$ (\eqref{eq:increasing_s1} implies that $s_1(x) < s_1(y)$ is impossible). It follows from \eqref{eq:increasing_s1} that $s_1(x) \ge s_1(z) \ge s_1(y)$ for all $z$ such that $z\in[y, x]_\cK$. However, we assumed above that $s_1(x) = s_1(y)$, which implies that $s_1(y) = s_1(z) = s_1(x)$ for all $z\in[y, x]_\cK$. It follows that $s_1(\cdot)$ is constant on the interval $[y, x]_\cK$, which has a nonzero Lebesgue measure since $x \gg_\cK y$. This is impossible since it is known that the level sets of $s_1 \in C^1(\cB(x^\ast))$ (i.e. the isostables) are of co-dimension $1$ in $\cB(x^\ast)$ (see~\cite{mauroy2013isostables}). Indeed, let $w_1$ be a left eigenvector of the Jacobian matrix corresponding to $\lambda_1$, then Hartman-Grobman theorem implies that $s_1(x)=w_1^T (x-x^\ast) + o(\|x-x^\ast\|)$ in the neighborhood of the equilibrium $x^\ast$ (cf.~\cite{Lan}), so that the isostables are locally homeomorphic to hyperplanes of co-dimension $1$ (\cite{mauroy2013isostables}). Every isostable can be obtained by backward integration starting from the neighborhood of $x^\ast$ and is therefore also of co-dimension $1$.\\
Finally, we note that we can obtain $v_1^T \nabla s_1(x^\ast) =1$ by multiplying $s_1$ with a proper positive constant. The above results imply $v_1^T \nabla s_1(x^\ast) \geq 0$. Since we know that $\nabla s_1(x^\ast)= w_1$, we have $v_1^T\nabla s_1(x^\ast) \neq 0$ and the result follows.\\
(ii) \emph{Necessity.} By Proposition~\ref{prop:trivial}, $\lambda_1$ is simple, real and negative, while $v_1\gg_\cK 0$. Due to the premise, $x\succ_\cK y$ implies that there exists $\tau_0$ such that $\phi(t,x) - \phi(t,y) \gg_\cK 0$ for all $t\ge \tau_0$. Moreover, (i) implies that $s_1(x)\ge s_1(y)$. Let $z = \phi(\tau_0,x)$ and $w = \phi(\tau_0, y)$. Since $z\gg_\cK w$, by (i) we have that $s_1(z) > s_1(w)$. This directly implies that $s_1(x) \ne s_1(y)$ by the property \eqref{eq:prop_eigenf} of $s_1(\cdot)$ and proves the claim. \\
\emph{Sufficiency.} It follows from \eqref{eq:expansion} that
\begin{equation}
\phi(t,x) - \phi(t,y) = e^{\lambda_1 t}\left(v_1 (s_1(x) - s_1(y)) +  \bar{R}(t) \right)\label{eq:phi-bound}
\end{equation}
with
\begin{equation*}
\begin{split}
\bar{R}(t) & = \sum \limits_{j=2}^n v_j (s_j(x)-s_j(y)) e^{(\lambda_j-\lambda_1) t}\\
& \qquad + e^{-\lambda_1 t} (R(t,x)-R(t,y)).
\end{split}
\end{equation*}
Since $v_1\gg_\cK 0$ it follows that $v_1(s_1(x) - s_1(y)) \gg_\cK 0$ for all $x\succ_\cK y$. Moreover, Proposition \ref{prop:Koopman_expansion} implies that
\begin{equation*}
0 = \lim_{t \rightarrow 0} \frac{|R(t,x)|}{\left\|\sum \limits_{j=1}^n v_j s_j(x) e^{\lambda_j t} \right\| } \geq  \lim_{t \rightarrow 0} \frac{e^{-\lambda_1 t}|R(t,x)|}{ \left\|\sum \limits_{j=1}^n v_j s_j(x) e^{(\lambda_j - \lambda_1) t} \right\| },
\end{equation*}
so that $\lim_{t \rightarrow 0} e^{-\lambda_1 t} |R(t,x)| = 0$ since $\lambda_j - \lambda_1 < 0$ for all $j>1$. Moreover, since $s_j(x)$ is finite for any given $x \in \cC \subseteq \cB(x^\ast)$,
we finally obtain that $\lim_{t \rightarrow 0} \bar{R}(t) = 0$. Then there exists $\tau_0\ge 0$ such that
\begin{gather}\label{eq:r-bound}
v_1(s_1(x) - s_1(y)) + \bar{R}(t) \gg_\cK 0\quad \forall t\ge \tau_0.
\end{gather}

The condition~\eqref{eq:phi-bound} together with \eqref{eq:r-bound} imply that
\begin{gather*}
\phi(t, x) - \phi(t,y) \gg_\cK 0\quad \forall x\succ_\cK y, \forall t\ge \tau_0,
\end{gather*}
which completes the proof.\\
(iii) The proof is straightforward.
\end{proof}
\begin{rem}Under the premise of Theorem~\ref{prop:ev-mon}, the condition $s_1(x) \ge s_1(y)$ for all $x \succeq_\cK y$ is equivalent to $\nabla s_1(x) \in \cK^\ast$, where $\cK^\ast$ is the dual cone to $\cK$, namely $\cK^\ast= \{y \in \R^n \bigl| y^T z \ge 0,\, \forall z \in \cK \}$. The condition $s_1(x) > s_1(y)$ for all $x \gg_\cK y$ is equivalent to $\nabla s_1(x) \in \inter(\cK^\ast)$.
\end{rem}
This result can be seen as a nonlinear extension of the Perron-Frobenius theorem (cf.~\cite{lemmens2012nonlinear}). If the assumptions in Theorem~\ref{prop:ev-mon} hold, it provides a complete description of strong eventually monotone systems, but only a necessary condition for eventual monotonicity. Surprisingly, this is different from classical monotonicity, since there exist necessary and sufficient conditions for monotonicity (Kamke-M\"uller conditions), but only sufficient conditions for strong monotonicity.
 
In the case of linear systems $\dot x = A x$, the points (i) and (ii) of Theorem \ref{prop:ev-mon} are reduced to Proposition~\ref{prop:ev-pos-dyn} since $s_1(x)=w_1^T x$ and $\nabla s_1(x)=w_1$, where $w_1$ is a left eigenvector of $A$ corresponding to $\lambda_1$. In other words, the statement of Proposition~\ref{prop:ev-pos-dyn} (in terms of eigenvectors and eigenvalues of $A$) is directly extended to nonlinear systems in Theorem~\ref{prop:ev-mon} (in terms of eigenfunctions and eigenvalues of the Koopman operator). Moreover, the proof of Theorem~\ref{prop:ev-mon} uses similar techniques as the proof of Proposition~\ref{prop:ev-pos-dyn}, which shows the power of the Koopman operator framework.

\begin{rem}
Using Theorem~\ref{prop:ev-mon} and the relation~\eqref{eq:def_eigenf}, it is now straightforward to show point (ii) in Proposition~\ref{prop:order_traject} for the case of strongly eventually monotone systems on a compact set $\cC$, provided that the assumptions in Theorem~\ref{prop:ev-mon} hold. Indeed, we have that $s_1(x)> 0$ for all $x\succ_\cK x^\ast$, hence $(f(x))^T\nabla s_1(x) =\lambda_1 s_1(x) < 0$. Now, since $\nabla s_1(x)\in\inter(\cK^\ast)$, it follows that $f(x)\not\succ_\cK 0$ for all $x\succ_\cK x^\ast$. 
\end{rem}
We conclude this subsection by a corollary of Theorem~\ref{prop:ev-mon}, which describes the geometric properties of the system in the basin of attraction. We introduce the following sets $\cB_\alpha = \left\{ x\in \cB(x^\ast) \Bigl| |s_1(x)| \le \alpha \right\}$, $\partial\cB_\alpha = \left\{ x\in\cB(x^\ast) \Bigl| |s_1(x)| = \alpha \right\}$. The level sets $\partial\cB_\alpha$ are called \emph{isostables} (\cite{mauroy2013isostables}) and contain the initial conditions of trajectories that converge synchronously towards the equilibrium. With $\alpha =\infty$, the set $\cB_\alpha$ is the basin of attraction, while $\partial \cB_\alpha$ is its boundary. We have the following result, whose proof can be found in Appendix~\ref{app:proofs}.

\begin{cor} \label{prop:level-set-proper} Consider that the system~\eqref{eq:sys} with $f\in C^2(\cD)$ has an equilibrium $x^\ast$ and satisfies Assumptions A1--A3. Assume that it is eventually monotone with respect to $\cK$ on the basin of attraction $\cB(x^\ast)$. Then the following relations hold for isostables defined with a monotone eigenfunction $s_1(x)$ and for all finite positive $\alpha$:\\
(i) $[z, w]_\cK\subset\cB_\alpha$ for any $w$, $z$ in $\cB_\alpha$ such that $z\preceq_\cK w$; \\
(ii) the level set $\partial\cB_\alpha$ does not contain points $z$, $w$ such that $z\gg_\cK w$. If the system is strongly eventually monotone, then $\partial\cB_\alpha$ does not contain points $z$, $w$ such that $z\succ_\cK w$.
\end{cor}

This result implies that, for strongly eventually monotone systems, the level sets $\partial\cB_\alpha$ of $s_1$ contain only incomparable points (with respect to some cone $\cK$). It directly follows from the fact that the eigenfunction $s_1$ is a monotone function.

\subsection{Relationship Between Eventual Monotonicity and Differential Positivity}
In this subsection we discuss the relation between eventually monotone systems and \emph{differentially positive} systems.
We need a few definitions to proceed. Assume that that the vector field $f$ is continuously differentiable and consider the so-called \emph{prolonged} dynamical system: 
\begin{align}
\label{sys:f-pr} &
\begin{array}{l}
\dot x = f(x),  \\
\dot{\delta x} = J(x) \delta x,
\end{array} 
\end{align} 
with $(x,\delta x) \in \mathbb{R}^n\times \mathbb{R}^n$ and where $J(x)$ stands for the Jacobian matrix of $f(x)$. We denote the differential of $\phi(t,x)$ with respect to $x$ by $\partial \phi(t,x)$. Then the prolonged system induces a flow $(\phi,\partial \phi)$ such that $(t,x,\delta x) \mapsto (\phi(t,x),\partial \phi(t,x) \delta x)$ is a solution of \eqref{sys:f-pr}. Following the definitions by~\cite{forni2015differentially}, we let \emph{a smooth cone field} $\cK(x)$ be defined as
\begin{gather*}
\cK(x) =\left\{ \delta x \in \R^n \Bigl| k_i(x, \delta x) \ge 0\, i = 1,\dots, m\right\},
\end{gather*}
where $\cK(x)$ is a positive cone for every $x\in\R^n$, and $k_i(\cdot,\cdot)$ are smooth functions.
\begin{defn}
The system~$\dot x = f(x)$ with $f\in C^1(\cD)$ is called differentially positive on $\cC\subseteq\cD$ with respect to the cone field $\cK(x)$ if the prolonged system leaves $\cK(x)$ invariant. Namely, for any $x\in\cC$
\begin{equation}
 \delta x \in\cK(x)\, \Rightarrow\, 
 \partial \phi(t,x)\delta x \in\cK(\phi(t,x)) \,\,  \forall t \ge 0 \,.
 \label{eq:diff-pos}
\end{equation}
The system is (uniformly) strictly differentially positive on $\cC\subseteq \cD$ if it is differentially positive on $\cC$ and if there exist a constant
$T > 0$ and a cone field $\cR(x) \subset \inter(\cK(x)) \cup \{0\}$  such that for all $x\in\cC$, $t \ge T$
\begin{equation}
\delta x \in\cK(x)\, \Rightarrow\, 
\partial \phi(t,x)\delta x \in\cR(\phi(t,x)).
\label{eq:strict-diff-pos}
\end{equation}
\end{defn}
Inspired by~\cite{mauroy2015operator}, we will use the following cone field based on the eigenfunctions of the Koopman operator:
\begin{multline}\label{ice-cream-cone-field}
\cK_{\alpha}(x)= \Bigl\{ y \in \R^n \Bigl|\,\,\sum\limits_{i =2}^n\, \alpha_i(x) \, |y^T \nabla s_i(x)|^2 \le \\(y^T \nabla s_1(x))^2,\,\,
y^T \nabla s_1(x) \ge 0\Bigl\}\,. 
\end{multline}
For every $x$ the set $\cK_\alpha(x)$ with smooth positive functions $\alpha_i(x)$ is a cone which appears to be the extension of~\eqref{ice-cream-cones} to the nonlinear case. By extending Proposition~\ref{prop:pos-cone} to nonlinear systems, we link differential positivity and eventual monotonicity in the following result (the proof is given in Appendix~\ref{app:proofs}).

\begin{prop}\label{prop:diff-pos-cone}
Consider that the system~\eqref{eq:sys} with $f\in C^2(\cD)$ has an equilibrium $x^\ast$ and satisfies Assumptions A1--A3, while $\lambda_i$ are the eigenvalues of the Jacobian matrix $J(x^\ast)$. Then: \\
(i) the system is strictly differentially positive on $\cC\subset\cB(x^\ast)$ with respect to $\cK_\alpha(x)$ with positive functions $\alpha_i(x)$ if and only if $\lambda_1$ is simple, real, negative and $\lambda_1 > \Re(\lambda_j)$ for all $j\ge 2$; \\
(ii) furthermore, the system is strongly eventually monotone on $\cC\subset\cB(x^\ast)$ with respect to $\Rnn^n$ if and only if $\lambda_1$ is simple, real, negative, and $\lambda_1 > \Re(\lambda_j)$ for all $j\ge 2$, there exist $\beta_i > 0$ and a function $\gamma(x) > 0$ such that $\cK_{\beta}(0) \subset \inter(\Rnn^n)\cup\{0\}$, and $\Rnn^n\subset \inter(\cK_\alpha(x))\cup\{0\}$ for all $x$ with $\alpha_i(\cdot) = \gamma(\cdot)$  for all $i$.
\end{prop}

We must point again to the similarities of the proofs of Proposition~\ref{prop:pos-cone} and Proposition~\ref{prop:diff-pos-cone}, although the latter is the nonlinear extension of the former. However, the implications of the results are slightly different. As discussed in Section~\ref{ss:cone-ev-pos}, in the linear setting, any eventually positive system is also positive (i.e. monotone) with respect to some cone $\cK$. In the nonlinear setting, any sufficiently smooth strongly eventually monotone system is differentially positive with respect to some cone field $\cK(x)$, which does not generally imply monotonicity with respect to some cone $\cK$. This is an indication that eventual monotonicity can offer a trade-off between monotonicity and differential positivity, as long as ordering properties of the solutions are concerned. To summarize, Proposition~\ref{prop:diff-pos-cone} provides conditions, based on the invariant cone field, under which a differentially positive system behaves asymptotically as a monotone system. 

\subsection{Eventual Monotonicity with Respect to a Cone}
It can be shown that the statement of Proposition~\ref{prop:diff-pos-cone} still holds when the orthant $\Rnn^n$ is replaced by a constant cone $\cK$. However, using Proposition~\ref{prop:diff-pos-cone} in practice requires computation of the cone fields, which is not an easy task. In the following result, we discuss an alternative way to check if the system is strongly eventually monotone with respect to a constant cone $\cK$, and how to find this cone. This result can be also seen as an extension of the results in Theorem~\ref{prop:ev-mon}. The proof is given in Appendix~\ref{app:proofs}.
\begin{cor} \label{prop:another-charact} 
Consider that the system~\eqref{eq:sys} with $f\in C^2(\cD)$ has an equilibrium $x^\ast$ and satisfies Assumptions A1--A3, while $\lambda_j$ are the eigenvalues of $J(x^\ast)$. Assume that $\lambda_1$ is simple, real, negative and $\lambda_1 > \Re(\lambda_j)$ for all $j\ge 2$. Then
\begin{gather}
v_1^T \nabla s_1(x) > 0\quad\forall x\in\cB(x^\ast) \label{cond:sem}
\end{gather}
with $v_1^T\nabla s_1(x^\ast) = 1$ if and only if there exists a cone $\cK$ with respect to which the system is strongly eventually monotone.
\end{cor}
\begin{rem}\label{rem:pos-cert}
The condition~\eqref{cond:sem} is equivalent to the existence of a cone $\cK$ such that $v_1\in\inter(\cK)$ and $\nabla s_1(x) \in \inter(\cK^*)$ for all $x\in \cB(x^\ast)$, where $\cK^\ast$ is the dual cone to $\cK$.
\end{rem}

For linear systems, the condition~\eqref{cond:sem} is always fulfilled, which is implicitly used in Corollary~\ref{prop:dom-eig}. 
Corollary~\ref{prop:another-charact} provides a (necessary and sufficient) certificate to determine whether there exists a cone with respect to which a system is strongly eventually monotone. We can verify the certificate by computing the gradient $\nabla s_1(x)$ through Laplace averages (see Section~\ref{sec:num_computation}). In addition, the condition $\nabla s_1(x) \in \inter(\cK^*)$ is equivalent to finding an order $\gg_\cK$ with respect to which no pair $x$, $y$ on the same isostable is such that $x\gg_\cK y$. Hence, Corollary~\ref{prop:another-charact} also offers a graphical tool (at least for planar systems) to find candidate cones $\cK$ with respect to which the system may be strongly eventually monotone. Furthermore, even if a system is strongly eventually monotone with respect to $\Rnn^n$, $\nabla s_1(x)$ can have slightly negative components for some $x$ due to numerical errors. Corollary~\ref{prop:another-charact} tackles this issue as well by providing a smaller candidate cone $\cK$. We cover this idea in detail on an example in Section \ref{sec:examples}.

This result also highlights another difference between eventually monotone and monotone systems with respect to arbitrary cones $\cK$. If we suspect that a system is monotone with respect to some cone $\cK$, there is no automatic way (to our best knowledge) to compute candidate cones $\cK$ with respect to which the system is potentially monotone. In contrast, Corollary~\ref{prop:another-charact} allows this in the case of (strongly) eventually monotone systems.

\section{Examples }\label{sec:examples}
\subsection{Computation of $s_1(x)$ and its Gradient}
\label{sec:num_computation}

The property of (strong) eventual monotonicity must be studied through the properties of the eigenfunction $s_1$ of the Koopman operator. A numerical method required to compute this eigenfunction is provided by the Laplace averages \eqref{eq:Laplace_av} evaluated along the trajectories of the system. Through the Laplace averages, the value of $s_1$ can be computed for some points (uniformly or randomly) distributed in the state space. The eigenfunction $s_1$  is then obtained by interpolation. It is important to note that, for a good numerical convergence, the observable $g$ used in \eqref{eq:Laplace_av} should be $g(x)=w_1^T (x - x^\ast)$, where $w_1$ is the left eigenvector of the Jacobian matrix at the equilibrium $x^\ast$, associated with $\lambda_1$.

As shown in~\cite{mauroy2015operator}, the gradient $\nabla s_1$ can also be directly obtained through Laplace averages. Considering the prolonged system \eqref{eq:diff-pos}, we have the Laplace average
\begin{equation}
\label{eq:Laplace_gradient}
\tilde{g}^{av}_{\lambda}(x,\delta x) = \lim \limits_{t\rightarrow \infty}\frac{1}{t}\int\limits_0^t (g\circ (\phi(s,x),\partial \phi(s,x) \delta x)) e^{-\lambda s} d s
\end{equation}
with some observable $g:\mathbb{R}^n \times \mathbb{R}^n \to \mathbb{R}$. The gradient of $s_1$ can be computed as
\begin{equation*}
\nabla s_1(x)^T=\left(\begin{array}{ccc} \tilde{g}^{av}_{\lambda_1}(x,e_1), & \dots, & \tilde{g}^{av}_{\lambda_1}(x,e_n)
\end{array} \right)\,,
\end{equation*}
with $g(x,\delta x)=w_1^T \delta x$ and where $e^k$ is the $k$-th unit vector. We also point out that the eigenfunctions can be estimated by linear algebraic methods (cf.~\cite{mauroy2014global}) and directly from data using so-called dynamic mode decomposition (DMD) methods (cf.~\cite{Schmid2010,Tu2014}). 

\subsection{A Three State Nonmonotone System} \label{ex:3d}
\begin{figure}
 \begin{subfigure}[t]{\columnwidth}  \centering
 \includegraphics[width = 0.75\columnwidth]{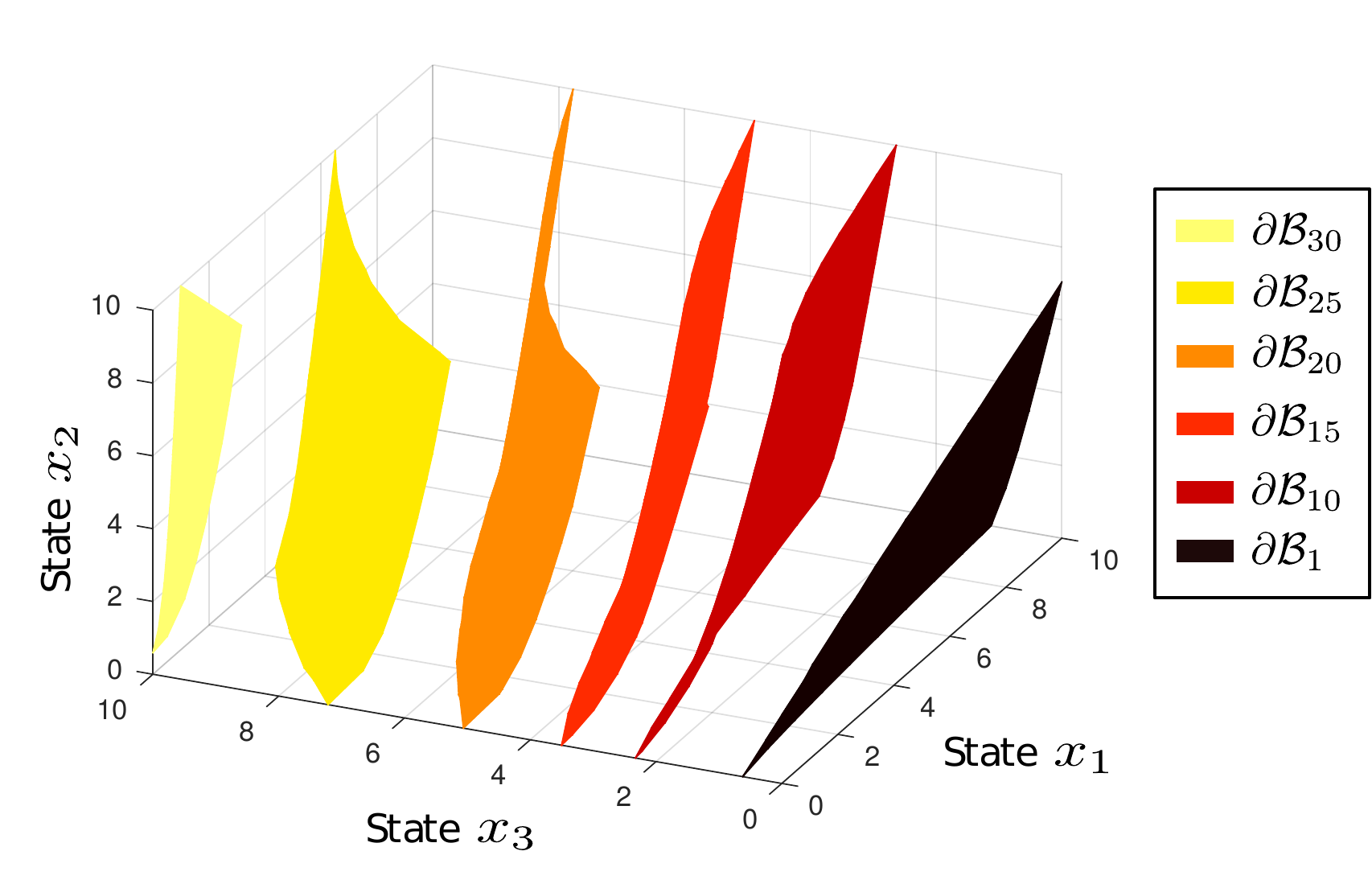}
 \caption{Level sets of $s_1(x)$.}
 \label{fig:three_state_nonmonotone}
 \end{subfigure}
  \begin{subfigure}[c]{\columnwidth}  \centering
\includegraphics[width = 0.7\columnwidth]{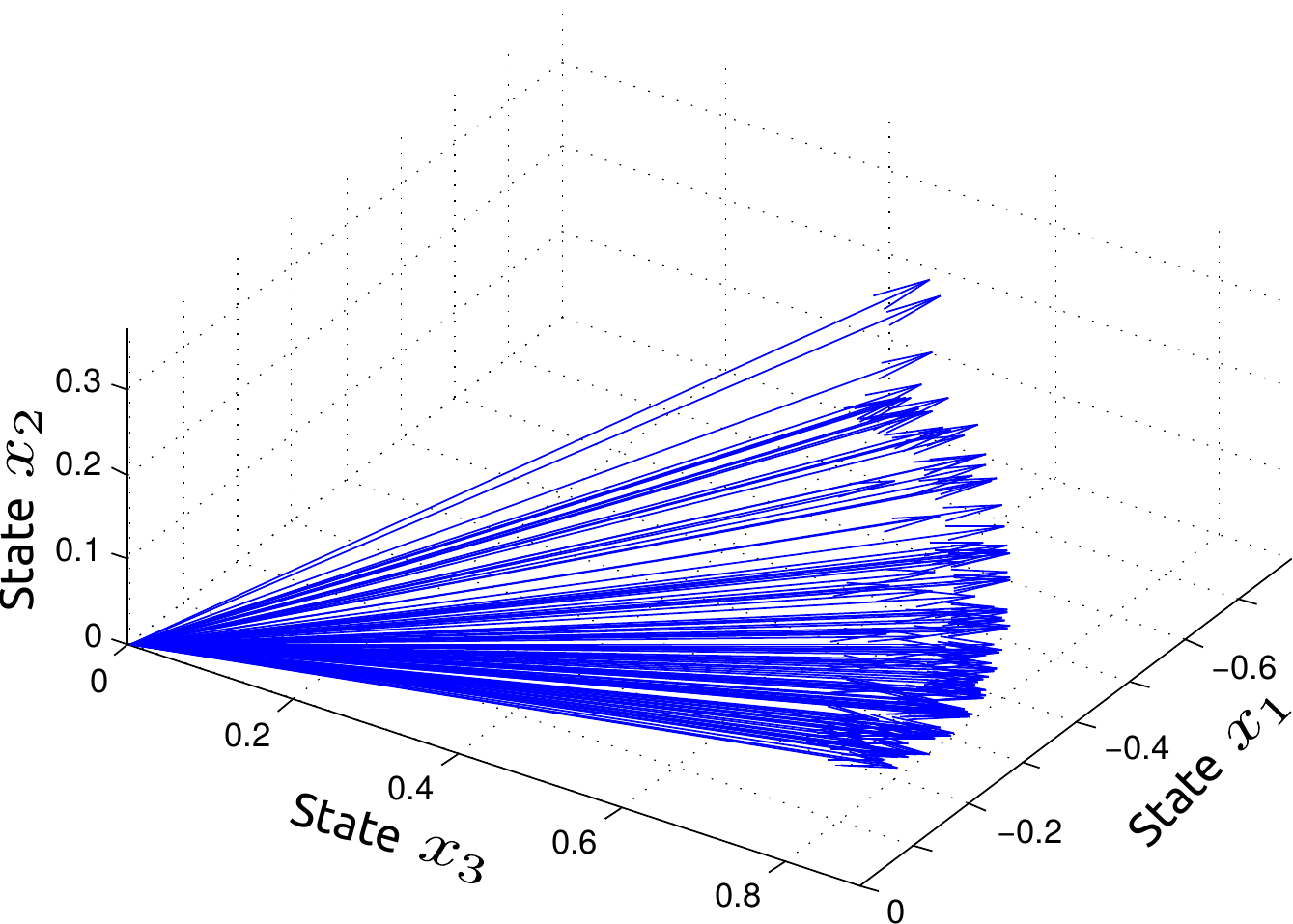}
  \caption{(Normalized) gradient of $s_1(x)$, i.e. $\nabla s_1(x)/\|\nabla s_1(x)\|$, for different states $x\in\cB(x^\ast)$. The gradient lies inside the orthant $\diag\{-1, 1, 1\}\Rnn^3$.}
  \label{fig:three_state_nonmonotone_gradient}
  \end{subfigure}
\caption{Figures describing a three state nonmonotone system (Example~\ref{ex:3d}).}
\end{figure}

Here, we revisit Example~\ref{ex:sing-pert} in Section~\ref{s:prel}. Using singular perturbation arguments, we verified that with small enough $\varepsilon$, the system is eventually monotone. Now consider the system with $\varepsilon = 1$. The second and third modes are equally fast, since the eigenvalues $\lambda_2$ and $\lambda_3$ are complex conjugate. However, the system is still strongly eventually monotone. We compute the eigenfunction $s_1$ of the Koopman operator in the basin of attraction $\cB(x^\ast)$ of the exponentially stable equilibrium $x^\ast=(3.1179, 0.2428, 1.4857)$ (Figure~\ref{fig:three_state_nonmonotone}). The computation of the gradient shows that $\nabla s_1(x)$ lies in the orthant $\diag\{-1, 1, 1\}\Rnn^3$ for all $x$ (Figure~\ref{fig:three_state_nonmonotone_gradient}). Moreover, we have $v_1=(-0.96, 0.07, 0.24) \in \diag\{-1, 1, 1\}\Rnn^3$. It follows that the condition~\eqref{cond:sem} is satisfied and, according to Corollary~\ref{prop:another-charact}, the system is strongly eventually monotone on $\cB(x^\ast)$ (with respect to that orthant $\diag\{-1, 1, 1\}\Rnn^3$). We can explain such a behavior as follows. In  Example~\ref{ex:sing-pert}, the term $h(x_1)$ is not consistent with monotonicity with respect to $\diag\{-1, 1, 1\}\Rnn^3$. However, the function is bounded between $0$ and $1$ for all $x_1$, so that its effect is not considerable enough to prevent eventual monotonicity.

\subsection{Gut Kinetics in Type I Diabetic Patients}\label{ex:gut}
A mathematical model of gut kinetics in type I diabetic patients is described by the following three state system (\cite{man2006system}):
\begin{align*}
\dot{Q}_{sto1}      & = -k_{2 1} Q_{sto1}, \\
\notag\dot{Q}_{sto2}& = -k_{empt}(z) Q_{sto2} + k_{2 1} Q_{sto1}, \\
\notag\dot{Q}_{gut} & =-k_{abs} Q_{gut} + k_{empt}(z) Q_{sto2},
\end{align*}
where $z=Q_{sto1}+Q_{sto2}$ and $k_{empt}(z)= k_{min} + \frac{k_{max} -k_{min}}{2}
(\tahn(\alpha(z-b))-\tahn(\beta (z-c)) + 2 )$.

According to a physiologic intuition, the more carbohydrates are ingested the higher the output $Q_{gut}$ should be under normal conditions (patient does not experience stress, does not exercise at the moment, etc.). Although the mathematical model is not monotone with respect to any orthant, it fulfills all the necessary conditions for eventual monotonicity. Indeed, the eigenfunction $s_1$ is strictly monotone in the state variables $Q_{sto1}$ and $Q_{sto2}$ (Figure \ref{fig:gut_kinetics}) and does not depend on the state $Q_{gut}$ (due to the cascade structure of the model). Moreover, $v_1=(0,0.94,0.34) \in \Rnn^3$.
\begin{figure}[t]
  \centering
\includegraphics[width = 0.65\columnwidth]{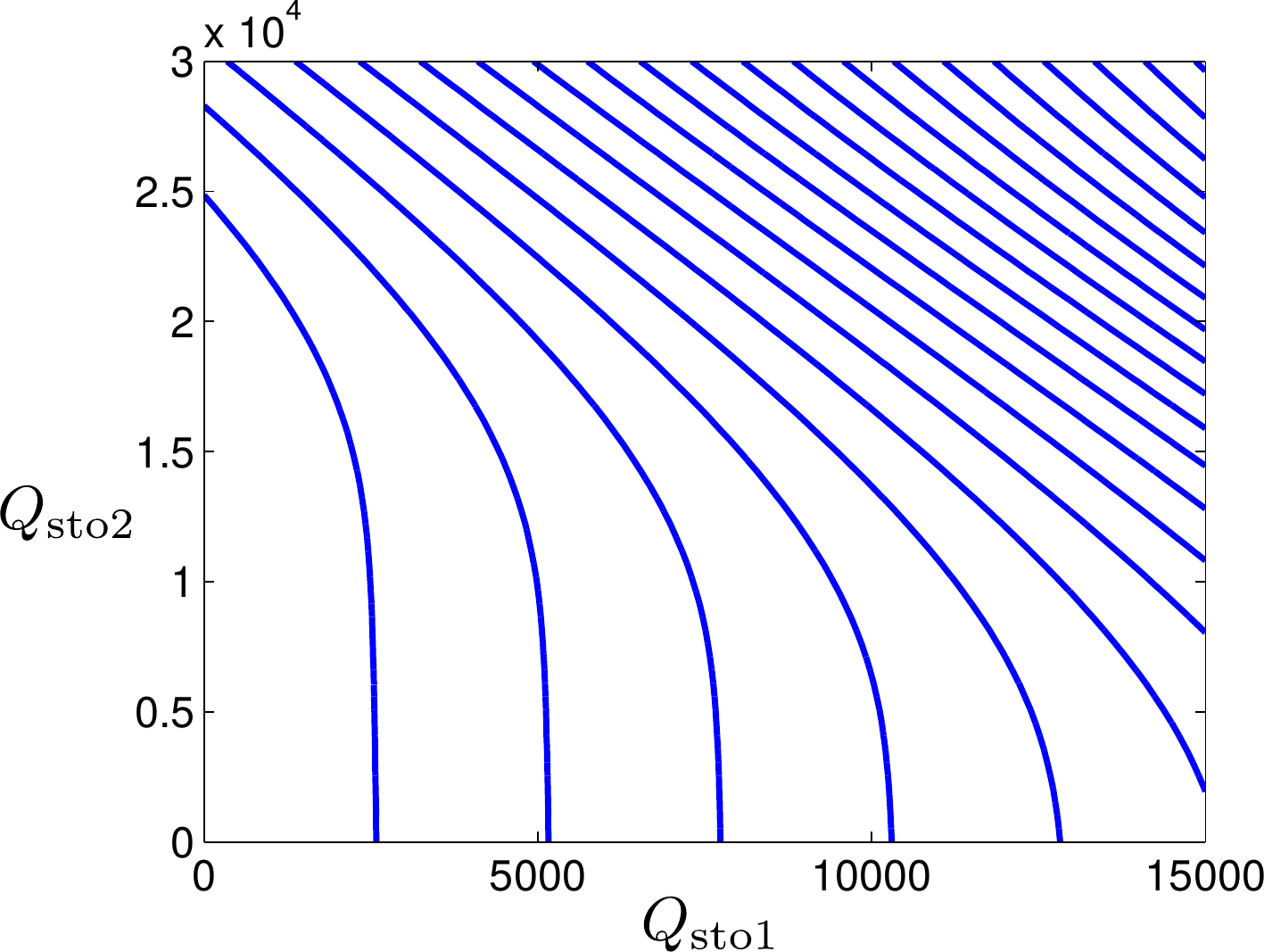}
  \caption{Level sets of $s_1$ (cross-section for a constant $Q_{gut}$) for the model of gut kinetics (Example~\ref{ex:gut}). The eigenfunction is strictly monotone in the two state variables $Q_{sto1}$ and $Q_{sto2}$.}
  \label{fig:gut_kinetics}
\end{figure}

All the eigenvalues of the Koopman operator are negative and distinct. However, since $v_1$ has a zero component in the $Q_{sto1}$ direction (and the gradient $\nabla s_1$ in the $Q_{gut}$ direction), the system is not strongly eventually monotone with respect to $\Rnn^3$. This is not surprising, since the Jacobian matrix is reducible for every $x\in\Rnn^3$. Nevertheless, numerical computations (using~\eqref{eq:Laplace_gradient}) show that $v_1^T \nabla s_1(x) >0$ for all $x \in \Rnn^3$, and Corollary~\ref{prop:another-charact} implies that the system is strongly eventually monotone with respect to some cone in $\R^3$. This example illustrates that even a system with a reducible Jacobian matrix can still be strongly eventually monotone albeit with respect to a different cone.

\subsection{Toxin-Antitoxin System}\label{ex:ta}
Consider the toxin-antitoxin system studied in~\cite{cataudella2013conditional}:
\begin{align*}
\dot T &= \frac{\sigma_T}{\left(1 + \frac{[A_f][T_f]}{K_0}\right)(1+\beta_M [T_f])} - \dfrac{1}{(1+\beta_C [T_f])} T \\
\dot A &= \frac{\sigma_A}{\left(1 + \frac{[A_f][T_f]}{K_0}\right)(1+\beta_M [T_f])} - \Gamma_A  A \\
\varepsilon [\dot A_f] &= A - \left([A_f] + \dfrac{[A_f] [T_f]}{K_T} + \dfrac{[A_f] [T_f]^2}{K_T K_{T T}}\right) \\ 
\varepsilon [\dot T_f] &= T - \left([T_f] + \dfrac{[A_f] [T_f]}{K_T} + 2 \dfrac{[A_f] [T_f]^2}{K_T K_{T T}}\right), 
\end{align*}
where $A$ and $T$ is the total number of toxin and antitoxin proteins, respectively, while $[A_f]$, $[T_f]$ is the number of free toxin and antitoxin proteins.
In~\cite{cataudella2013conditional}, the authors considered the model with $\varepsilon = 0$. In order to simplify our analysis we set $\varepsilon = 10^{-6}$. With the parameters
$\sigma_T = 166.28$, $K_0 = 1$, $\beta_M = \beta_c =0.16$, $\sigma_A = 10^2$
$\Gamma_A = 0.2$, $K_T = K_{TT} = 0.3$, the system is bistable with two exponentially stable equilibria:
\begin{gather*}
 x^\ast = \begin{pmatrix}  27.1517 &  80.5151 & 58.4429 & 0.0877  \end{pmatrix} \\ 
 x^{\bullet} = \begin{pmatrix} 162.8103 & 26.2221  &  0.0002 & 110.4375 \end{pmatrix}. 
\end{gather*}
We consider the basin of attraction of the equilibrium $x^\ast$ and compute the level sets of $s_1(x)$. Because of the slow-fast dynamics of the model, the eigenfunction is (almost) constant in the variables $[A_f]$ and $[T_f]$. Its values in the cross-section $([A_f],[T_f])=(x^\ast(3),x^\ast(4))$ are shown in Figure~\ref{fig:tat-cl}. As the reader may notice, some level sets in the cross-section are consistent with eventual monotonicity with respect to the orthant $\cK = (1 ,-1) \Rnn^2$. There are however level sets $\partial\cB_\alpha$ of $s_1(x)$, which contain comparable points and hence these level sets are not consistent with eventual monotonicity with respect to $\cK$. However, every level set contains incomparable points with respect to a cone $\tilde{\cK}$ depicted in red in Figure~\ref{fig:tat-cl}. According to Corollary~\ref{prop:level-set-proper}, this suggests that the system might be strongly eventually monotone with respect to a cone (whose projection on the cross-section is $\tilde{\cK}$). Indeed, numerical computations (using~\eqref{eq:Laplace_gradient}) show that $v_1^T \nabla s_1(x) >0$ for all $x \in \Rnn^4$, so that Corollary~\ref{prop:another-charact} can be applied. Note that the reduced order model (i.e. with $\varepsilon = 0$) is only locally monotone (around $x^\ast$, $x^\bullet$) with respect to the orthant $\cK = (1 ,-1) \Rnn^2$, which we verify by numerically computing the Jacobian matrix of the reduced order system.

 \begin{figure}[t]
 	\centering
 	\includegraphics[width = 0.65\columnwidth]{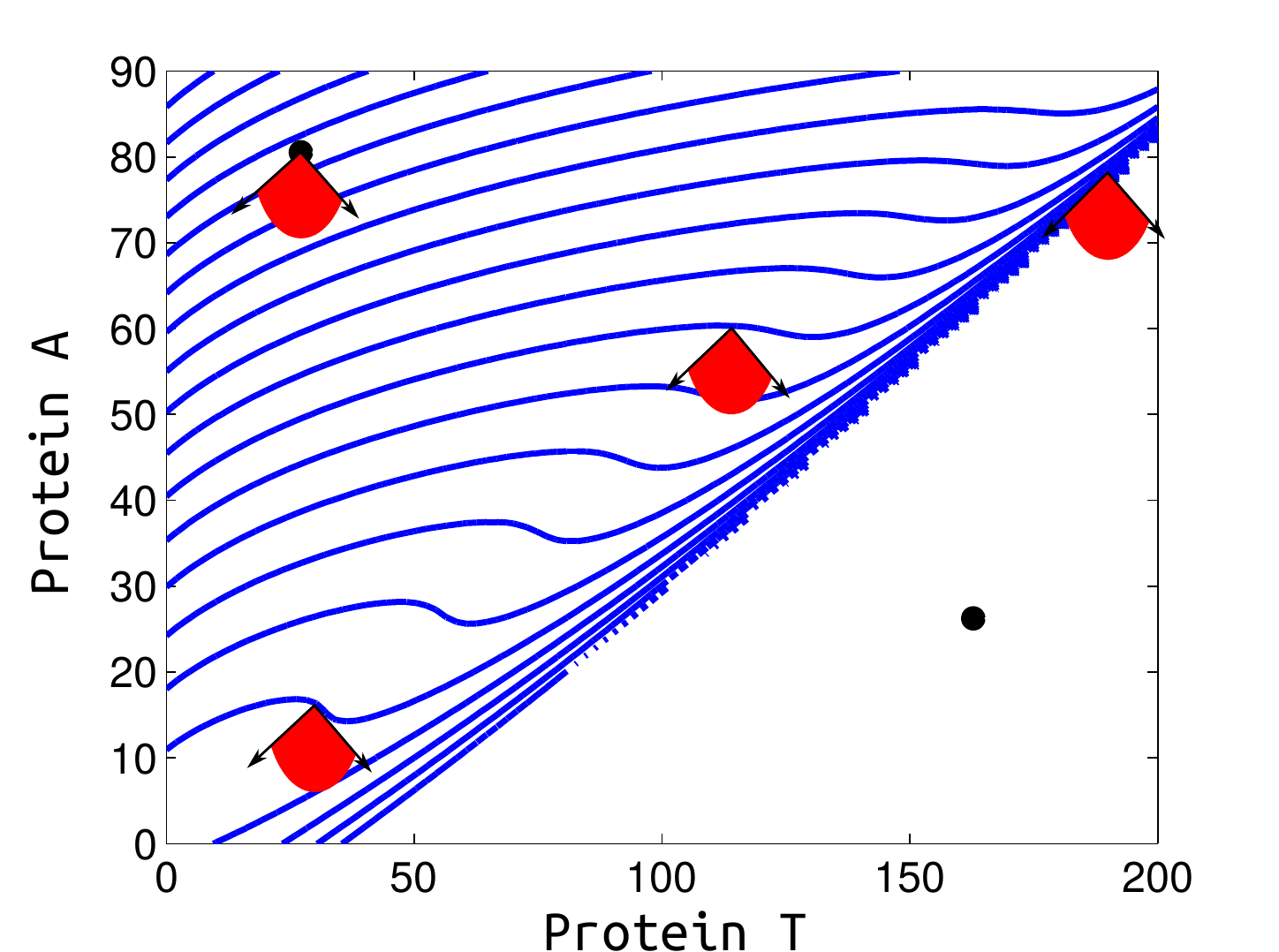}
 	\caption{Level sets of $s_1(x)$ for the toxin-antitoxin system (Example~\ref{ex:ta}) in the cross-section $([A_f],[T_f])=(x^\ast(3),x^\ast(4))$. The system is strongly eventually monotone with respect to a cone, whose projection in the cross-section is the red cone $\tilde{\cK}$ shown on the picture.}
 	\label{fig:tat-cl}
 \end{figure}

\subsection{FitzHugh-Nagumo Model}
The excitable FitzHugh-Nagumo model (\cite{fitzhugh1961impulses, nagumo1962active}) is described by the following equations
\begin{gather*}
\dot v = - w -v(v-1)(v-a) + I, \\
\dot w = \epsilon (v -\gamma w).
\end{gather*}
We take $a = 1$, $I = 0.05$, $\epsilon = 0.08$, $\gamma = 1$, which results in a system with a Jacobian matrix at the equilibrium having simple, real, negative eigenvalues. This implies that the system is locally eventually monotone around the exponentially stable equilibrium. However, it is not globally eventually monotone with respect to any cone $\cK\in\R^2$, which is verified by the fact that some level sets of $s_1$ contain comparable points in all possible orderings (see Figure~\ref{fig:fitzhugh}).

\begin{figure}[t]
  \centering
\includegraphics[width = 0.75\columnwidth]{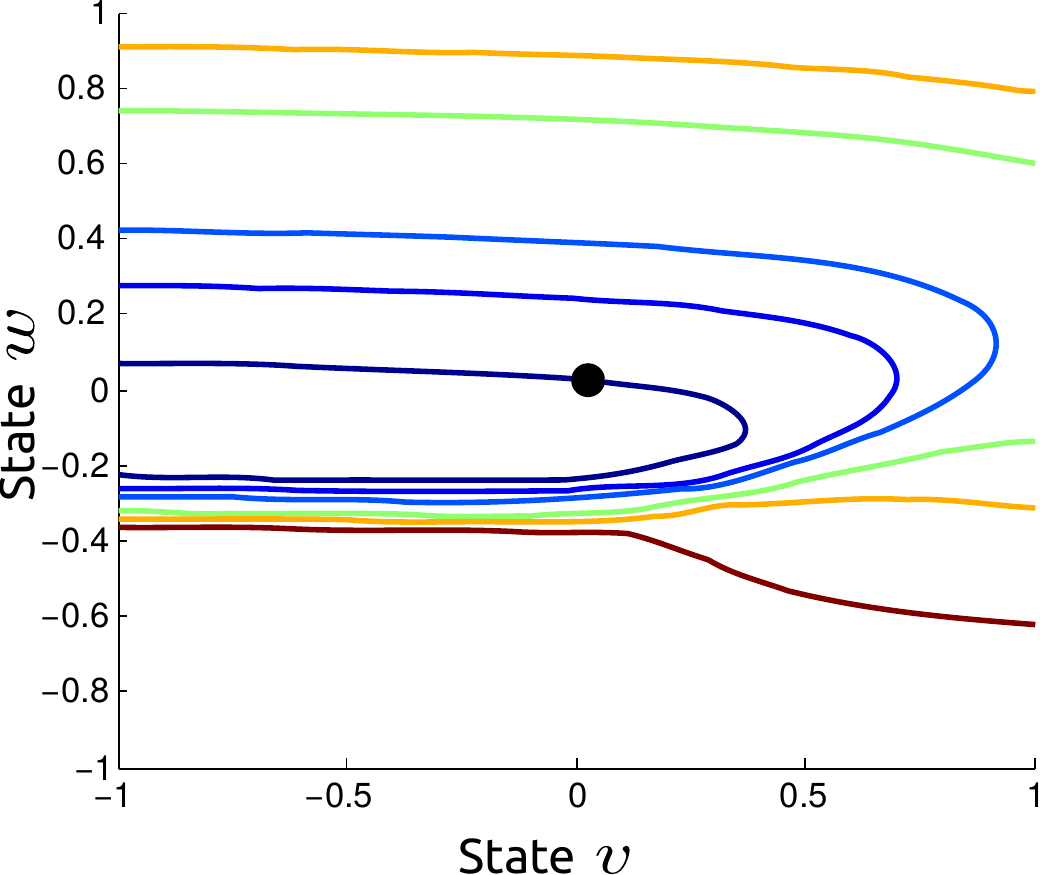}
  \caption{Level sets of $s_1(x)$ for the FitzHugh-Nagumo model with $a = 1$, $I = 0.05$, $\epsilon = 0.08$. The dot is the exponentially stable equilibrium, around which the system is locally eventually monotone. The system, however, is not globally eventually monotone with respect to any cone in $\R^2$.}
  \label{fig:fitzhugh}
\end{figure}
\section{Discussion and Conclusion}\label{s:con}

In this paper, we have provided a characterization of (strongly) eventually monotone systems using spectral properties of the so-called Koopman operator. Our results indicate that eventually monotone systems possess many asymptotic properties of monotone systems, possibly providing a valuable theoretical generalization of monotonicity. We present examples of systems, which cannot be confirmed to be monotone, but are strongly eventually monotone. The examples describe biological and biomedical processes, showing that there are potentially many applications of eventual monotonicity. Moreover, the spectral operator-theoretic framework considered in this paper offers a numerical tool to compute candidate cones with respect to which the system may be strongly eventually monotone. To our best knowledge, no such tool exists for monotone systems.

Strong eventual monotonicity has applications in model reduction. If a full order system is strongly eventually monotone, then there is a possibility that model reduction of fast states leads to a monotone system. This can lead to model reduction methods enforcing monotonicity on a reduced order dynamical system. Moreover, the results by~\cite{wang2008singularly} are based on what we call strong eventual monotonicity and our certificates provide a tool, which facilitates the application of~\cite{wang2008singularly} to a broad class of systems.

The main drawback of our theoretical development is the absence of a polynomial-time certificate for eventual monotonicity. Since we have derived a positivity certificate for eventual monotonicity, a polynomial-time version of this certificate could potentially be obtained through sum-of-square techniques (\cite{papachristodoulou2013sostools}). If we can certify that the system is strongly eventually monotone, then we can compute its basins of attraction with a high accuracy as discussed in~\cite{sootla2016basins,sootla2016geometry}.

We note that our work may potentially be related to the recent work by~\cite{daners2016eventually}, where the authors study eventually positive semigroups of linear operators in a general context. We leave it for future research. Finally, we aim at extending the concept of eventual monotonicity to \emph{open} or \emph{control} systems, which may lead to simple control strategies as the ones described in~\cite{sootla2015pulsesacc}.

\bibliography{bibl_koopman}
\appendix
\section{Proofs}\label{app:proofs}
For the proof of Proposition~\ref{prop:ev-pos-dyn}, we need the following result from~\cite{noutsos2008reachability}.
\begin{prop} \label{prop:exp-pos}
	Let $A\in\R^{n\times n}$. Then: \\
	(i) Let there exists a $\tau_0\ge 0$ such that for all $t \ge \tau_0$, the matrix $e^{t A}$ is nonnegative, then there exists a scalar $a$ such that $A + a I$ is a $\wpfn$ matrix; \\
	(ii) There exists a $\tau_0\ge 0$ such that for all $t \ge \tau_0$, the matrix $e^{t A}$ is positive if and only if there exists a scalar $a$ such that $A + a I$ is a $\pfn$ matrix.
\end{prop}
\begin{proof-of}{Proposition~\ref{prop:ev-pos-dyn}}
	(i) If the flow $\phi(t, x) = e^{A t} x$ is nonnegative for all $t\ge \tau_0$ for any nonnegative $x$, then $e^{A t}$ is nonnegative for all $t\ge \tau_0$. By Proposition~\ref{prop:exp-pos} there exists a scalar $a$ such that $A + a I$ is a $\wpfn$ matrix. This implies that there exist nonnegative right and left eigenvectors $v_1$ and $w_1$ corresponding to a real $\tilde \lambda_1$ such that $\tilde \lambda_1 > |\tilde \lambda_i|$ for all $i\ge 2$, where $\tilde \lambda_i$ are the eigenvalues of $A + a I$. It is straightforward to show that $\tilde \lambda_i = \lambda_i + a$, from which it follows that $\lambda_1$ is real and there exists no eigenvalue $\lambda_i$ such that $\lambda_1 = \Re(\lambda_i)$ and $\Im(\lambda_i) \ne 0$. Furthermore $v_1$, $w_1$ are also eigenvectors of $A$ corresponding to $\lambda_1$.\\
	(ii) \emph{Necessity.} If the flow $\phi(t, x) = e^{A t} x$ is positive for all $t\ge \tau_0$ for any nonnegative, nonzero $x$, then $e^{A t}$ is positive for all $t\ge \tau_0$. By Proposition~\ref{prop:exp-pos} there exists a scalar $a$ such that $A + a I$ is a $\pfn$ matrix. As in the point (i), we can show that $\lambda_1$ is simple, and the right and left eigenvectors $v_1$ and $w_1$ corresponding to  $\lambda_1$ can be chosen to be positive. Furthermore, $\lambda_1 > \Re(\lambda_j)$ for all $j\ge 2$.\\
	\emph{Sufficiency.} 
	Let $v_i$, $w_i$ be the right and left eigenvectors corresponding to the eigenvalues $\lambda_i$. Let $v_1$, $w_1$ be positive and $\lambda_1$ be real and $R(t,x) = \sum\limits_{i =2}^n e^{(\lambda_i -\lambda_1)t} v_i w_i^T x$. Then we have
	\[
	\phi(t, x) = e^{A t} x = e^{\lambda_1 t} \left(v_1 w_1^Tx + R(t,x)\right)\,.
	\]
	Since $v_1$, $w_1$ are positive and $\Re(\lambda_i) < \lambda_1$ for all $i\ge 2$, there exists a time $\tau_0$ such that $R(t,x)  \ll v_1 w_1^T x$ for all $t\ge\tau_0$. Hence we have that $\phi(t,x) \gg 0$.
\end{proof-of}
\begin{proof-of}{Proposition~\ref{prop:pos-cone}}
	(i) Let $y = e^{A t}x$ for $t>0$, then 
	\begin{align*}
	& (w_1^T y)^2 - \sum\limits_{i =2}^n \alpha_i |w_i^T y|^2 \\
	& \qquad  = (w_1^T e^{A t} x)^2   - \sum\limits_{i =2}^n\alpha_i |w_i^T e^{A t} x|^2\\
	& \qquad = e^{2 \lambda_1 t} (w_1^T x)^2 - \sum\limits_{i =2}^n\alpha_i |e^{\lambda_i t}|^2  |w_i^T x|^2 \\
	& \qquad = e^{2 \lambda_1 t} \Bigl( (w_1^T x)^2 - 
	\sum\limits_{i =2}^n\alpha_i |e^{(\lambda_i-\lambda_1) t}|^2 |w_i^T x|^2\Bigl).
	\end{align*}
	Since $\lambda_1 >\Re(\lambda_i)$ for all $i>1$ we have that $|e^{(\lambda_i-\lambda_1) t}|^2 < 1$ for all $t>0$, which in turn implies that
	\begin{multline*}
	(w_1^T y)^2 - \sum\limits_{i =2}^n \alpha_i |w_i^T y|^2 \ge \\e^{2 \lambda_1 t} \left( (w_1^T x)^2 - \sum\limits_{i =2}^n\alpha_i |w_i^T x|^2\right), 
	\end{multline*}
	and $y = e^{At} x$ belongs to $\cK_\alpha$ if $x$ does.\\
	(ii) \emph{Necessity.} Since $\cK_0 = \{y \in\R^n | w_1^T y \ge 0\}\supset\inter(\Rnn^n)\cup\{0\}$, there exist $\gamma_i$ small enough that $\cK_0 \supset\cK_\gamma \supset\inter(\Rnn^n)\cup\{0\}$. Similarly, there exist $\beta_i$ large enough such that $\inter(\Rnn^n)\cup\{0\}\supset\cK_\beta\supset\cK_\infty = \{y \in\R^n | y = \Delta v_1, \Delta >0 \}$. \\
	\emph{Sufficiency.} We have that $\cK_\infty \subset \cK_{\beta} \subset \inter(\Rnn^n)\cup\{0\}$, hence the vector $v_1$ is contained in $\cK_{\beta}$, which entails that the condition $\cK_{\beta} \subset \inter(\Rnn^n)\cup\{0\}$ ensures positivity of the eigenvector $v_1$. Now let $ \Rnn^n \subset \inter(\cK_{\gamma})\cup\{0\}$ for a positive $\gamma$. Hence the scalars $w_1^T e^j$ (where $e^j$ is $j$-the unit vector) are positive for all $j$, and $w_1$ is positive. Taking into account the arguments above, we conclude that the system $\dot x = A x$ is strongly eventually positive.
\end{proof-of}

\begin{proof-of}{Corollary~\ref{prop:dom-eig}}
	Let $v$ and $w$ be the right and left eigenvectors corresponding to the dominant eigenvalue $\lambda_1$ of $A$. According to Proposition~\ref{prop:pos-cone}, we need to show that there exists an invertible matrix $S$ such that $S^{-1} v$ and $S^T w $ are positive. Without loss of generality, we assume that the first entry of $w$ is nonzero. We can find a transformation $S$ such that $w^T S  = \bfone^T/n$ and $S \bfone  = v$ as follows
	\begin{gather*}
	S =  I_n/n + \begin{pmatrix}
	v -  \bfone/n & 0_{n\times n-1}
	\end{pmatrix}+\begin{pmatrix}
	\frac{(\bfone - w)^T}{w(1) n} \\
	0_{n-1\times n}
	\end{pmatrix} + S_0,
	\end{gather*}
	where $I_n$ is the $n\times n$ identity matrix, $0_{k\times m}$ is the $k\times m$ zero matrix, and $S_0$ is a zero matrix except for one entry, where $S_0(1,1) = (-1/w(1) +  w^T \bfone/(w(1) n))$. We verify the claim by direct calculations: 
	\begin{multline*}
	S \bfone =  \bfone/n + v - \bfone/n + \frac{(\bfone - w)^T}{w(1) n} \bfone \\+ (-1/w(1) +  w^T \bfone/(w(1) n)) = v
	\end{multline*}
	Similarly 
	\begin{multline*}
	w^T S =  w^T/n  + \begin{pmatrix}
	w^T v -  \frac{w^T\bfone}{n} & 0_{1\times n-1}
	\end{pmatrix} + \\ \frac{(\bfone - w)^T}{n} + 
	\begin{pmatrix}
	-1 + \frac{w^T \bfone}{n} & 0_{1\times n-1}
	\end{pmatrix}  =  \bfone^T /n
	\end{multline*}
	In this case new dominant eigenvectors are $\widetilde v = \bfone $ and $\widetilde w = \bfone/n$.  Since the eigenvalues do not change under the similarity transformation, the dominant eigenvalue of $S^{-1} A S$ is simple and real. The second statement is straightforward.
\end{proof-of}

\begin{proof-of}{Proposition~\ref{prop:order_traject}}
(i) The system is eventually monotone, therefore for any $\delta x\succ_\cK 0$, $y\in \cC$ and a sufficiently small $h\in \Rp$ we have that $\phi(t,y+ h \delta x)\succ_\cK \phi(t,y)$ for $t\ge\tau_0$. With $h\rightarrow 0$ we get that for all $t\ge \tau_0$ the vector $\partial \phi(t,x)\delta x$ belongs to $\cK$, where $\partial \phi(t,x)$ denotes the Jacobian of $\phi(t,x)$ with respect to $x$. Let $f(x)\succeq_\cK 0$, then $ f(\phi(t, x)) = d \phi(t,x)/ dt = \partial \phi(t, x) \dot x= \partial \phi(t, x) f(x)$, which implies that $f(\phi(t, x))$ belongs to $\cK$ for all $t\ge \tau_0$. Therefore $\phi(t+\Delta t,x)\succeq_\cK \phi(t,x)$ for all $\Delta t>0$ and $t\ge \tau_0$. \\
(ii)According to (i), $f(x)\succ_\cK 0$ implies that $\phi(t+\Delta t,x)\succ_\cK \phi(t,x) \succ_\cK x^\ast$ for $\Delta t>0$ and $t\ge \tau_0$. This contradicts the fact that $x\in \cB(x^\ast)$ and $\phi(t,x) \rightarrow x^\ast$. Therefore, $f(x)\not \succ_\cK 0$ and $\phi(t,x)\not \succ_\cK x$ for all $x\succ_\cK x^\ast$ and all $t>0$.\\
(iii) We will show the result by contradiction. Let $w$, $z$ belong to $\cB(x^\ast)$, let $y\in[z, w]_\cK$ and $y \notin \cB(x^\ast)$. Without loss of generality assume that $y$ belongs to the boundary of $\cB(x^\ast)$. Therefore the flow  $\phi(t, y)$ is on the boundary of $\cB(x^\ast)$. Let the distance between $x^\ast$ and this boundary be equal to $\varepsilon$. There exists a time $T_1$ such that for all $t>T_1$ the following inequalities hold $\|x^\ast - \phi(t, w)\|_2 < \varepsilon/2$,  $\|x^\ast - \phi(t, z)\|_2 < \varepsilon/2$.
Moreover, there exists a time $T_2>T_1$ such that for all $t>T_2$ and all $\xi$ in the interval $[\phi(t, z), \phi(t, w)]_\cK$ 
we have $\|x^\ast - \xi\|_2 < \varepsilon/2$. Now build a sequence $\{y^n\}_{n=1}^{\infty}$
converging to $y$ such that all $y^n$ lie in $\cB(x^\ast)$ and also lie in $[z, w]_\cK$. Due to eventual monotonicity on $\cB(x^\ast)$, for all $n$ and $t>\tau_0$ we have $\phi(t, y^n)\in [\phi(t, z), \phi(t, w)]_\cK$. Let $T_3 = \max(T_2,\tau_0)$ and note that for all $t>T_3$, we also have that $\|x^\ast - \phi(t, y^n)\|_2 < \varepsilon/2$. Since the sequence $\{y^n\}_{n=1}^{\infty}$ converges to $y$, by continuity of solutions, for all $t>T_3$ we have $\|x^\ast - \phi(t, y)\|_2 \le \varepsilon/2$, which contradicts $\|x^\ast - \phi(t, y)\|_2 \ge \varepsilon$ for all $t>0$.
\end{proof-of}
\begin{proof-of}{Proposition~\ref{prop:Koopman_expansion}}
It follows from Theorem 2.3 in \cite{Lan} that there exists a $C^1$ diffeomorphism $y = h(x)$ such that $h(\phi(t,x)) = e^{J(x^\ast) t} h(x)$ for all $x \in \cB(x^\ast)$, $h(x^\ast) = 0$, and the Jacobian matrix of $h$ at $x^\ast$ satisfies $(\partial h/\partial x)_{x=x^\ast} = I$. Considering the first order Taylor expansion of $h^{-1}$ around $0$, we obtain
\begin{equation*}
x = x^\ast  + \left( \frac{\partial h^{-1}}{\partial y} \right)_{y=0} \, y + o(\|y\|) =  x^\ast + y + o(\|y\|)
\end{equation*}
where we used the fact that the Jacobian matrix of $h^{-1}$ satisfies $(\partial h^{-1}/\partial y)_{y=0} = I$. Moreover, $C^1$ eigenfunctions of the Koopman operator are given by $s_{j}(x) = w_j^T h(x)$ and are associated with the eigenvalues $\lambda_j$. Indeed, we have
\begin{equation*}
\begin{split}
U^t s_j(x) & = w_j^T h(\phi(t,x))  = w_j^T e^{J(x^\ast) t} h(x)  =   e^{\lambda_j t} w_j^T h(x)  \\
           &= e^{\lambda_j t} s_j(x)
\end{split}
\end{equation*}
since the eigenvectors of $J(x^\ast)$ are eigenvectors of $e^{J(x^\ast) t}$. Equivalently, $s(x) = V^{-1} y$ where $s(x) = (s_1(x) \ \cdots \ s_n(x))^T$ and $V$ is a $n \times n$ matrix whose columns are the eigenvectors $v_j$. It follows that
\begin{equation}
\label{eq:expan_x}
\begin{split}
x & = x^\ast  + V s + o(\|V s\|) \\
& = x^\ast  +  \sum \limits_{j=1}^n v_j s_j(x) + o\left(\left\|\sum \limits_{j=1}^n v_j s_j(x) \right\|\right)
\end{split}
\end{equation}
and
\begin{equation*}
\phi(t,x) = x^\ast + \sum \limits_{j=1}^n v_j s_j(x) e^{\lambda_j t} + o\left(\left\|\sum \limits_{j=1}^n v_j s_j(x) e^{\lambda_j t}  \right\|\right) .
\end{equation*}
Finally, considering the Jacobian of \eqref{eq:expan_x} at $x^\ast$, we obtain
\begin{equation*}
I = V \ \left(\frac{\partial s}{\partial x}\right)_{x=x^\ast} = \left(\frac{\partial s}{\partial x}\right)_{x=x^\ast} \ V
\end{equation*}
so that $v_j^T \nabla s_j(x^*) = 1$. This concludes the proof.
\end{proof-of}

\begin{proof-of}{Corollary~\ref{prop:level-set-proper}}
(i) Since the points $w$, $z$ are in $\cB(x^\ast)$, the interval $[z, w]_\cK$ is in $\cB(x^\ast)$. Let $y$ belong to the interval $[z, w]_\cK$, with $|s_1(w)|$, $|s_1(z)|\le \alpha$. By Theorem~\ref{prop:ev-mon}, we have $s_1(w) \ge s_1(y) \ge s_1(z)$. Therefore we have two possibilities, either $|s_1(y)| \le |s_1(w)|$ or $|s_1(y)| \le |s_1(z)|$. In both cases, $|s_1(y)| \le \alpha$.\\
(ii) Let there exist $w$, $z$ in $\left\{ x \in \R^n \Bigl| s_1(x) = \alpha\right\}$ such that $w\gg_\cK z$ for some $\alpha\in\R$. We have that $s_1(w) = s_1(z)$, but according to Theorem~\ref{prop:ev-mon}, $w\gg_\cK z$ implies that $s_1(w) > s_1(z)$. Hence no such $w$ and $z$ exist. The second part of the statement is proved in a similar manner.
\end{proof-of}
\begin{proof-of}{Proposition~\ref{prop:diff-pos-cone}}
(i) Our result is based on a similar proposition in~\cite{mauroy2015operator}, therefore we only need to prove a result similar to Proposition~1 in~\cite{mauroy2015operator}. Using the equality $(\nabla s_i(\phi(t,x)))^T \partial \phi(t,x) \delta x = e^{\lambda_i t} (\nabla s_i(x))^T \delta x$, we get
\begin{align*}
& ( (\nabla s_1(\phi(t,x))^T \partial \phi(t,x) \delta x)^2    \\
&\qquad -\sum\limits_{i =2}^n \alpha_i(x) |(\nabla s_i(\phi(t,x)))^T \partial \phi(t,x) \delta x|^2  \\
&=e^{2 \lambda_1 t} ((\nabla s_1(x))^T \delta x)^2 -  \sum\limits_{i =2}^n \alpha_i(x) |e^{ \lambda_i t}|^2  |(\nabla s_i(x))^T\delta x|^2   \\
&=e^{2 \lambda_1 t} \Bigg( ((\nabla s_1(x))^T \delta x)^2 \\
&\qquad -\sum \limits_{i =2}^n \alpha_i(x) |e^{(\lambda_i-\lambda_1) t}|^2 |(\nabla s_i(x))^T \delta x|^2\Bigg)   \\
&\ge e^{2 \lambda_1 t} \left( ((\nabla s_1(x))^T \delta x)^2 - \sum \limits_{i =2}^n \alpha_i(x) |(\nabla s_i(x))^T \delta x|^2\right) \\
&\ge 0,
\end{align*}
for all $x\in\cB(x^\ast)$ and $\delta x \in \cK(x)$. The rest of the proof is identical to the proof of Proposition~3 in~\cite{mauroy2015operator}. \\
(ii) \emph{Necessity.} 
The inclusion $\cK_{\beta}(0) \subset \inter(\Rnn^n)\cup\{0\}$ follows from Hartman-Grobman theorem and Proposition~\ref{prop:pos-cone}.
Now since $\nabla s_1(x) \in \inter(\Rnn^n)$, there exists $\gamma(x)$ with sufficiently large values such that $\Rnn^n \subset \inter(\cK_\alpha(x))\cup\{0\}$ with $\alpha_i(x)=\gamma(x)$. \\
\emph{Sufficiency.} Since $\lambda_1$ is simple and real, we only need to show that $v_1$ and $\nabla s_1(x)$ are positive for all $x$. Using Hartman-Grobman theorem, as in the linear case, $\cK_{\beta}(0) \subset \inter(\Rnn^n)\cup\{0\}$ implies that $v_1 \gg 0$. If $\Rnn^n \subset \inter(\cK_\alpha(x))\cup\{0\}$ with $\alpha_i(x) =\gamma(x)$ for some function $\gamma(\cdot)$, then  $y^T \nabla s_1(x) > 0$ for all $y\in\Rnn^n\backslash\{0\}$ and all $x$. Therefore $\nabla s_1(x) \gg 0$ for all $x\in\cB(x^\ast)$. 
\end{proof-of}

\begin{proof-of}{Corollary~\ref{prop:another-charact}} \emph{Necessity.} According to Theorem~\ref{prop:ev-mon}, we have that $v_1 \gg_\cK 0$ and $\nabla s_1(x) \gg_\cK 0$ for all $x\in \cC$. The result follows from the definition of the dual cones. \\
\emph{Sufficiency.} According to Remark~\ref{rem:pos-cert}, there exists a cone $\cK$ such that $v_1\in \inter(\cK)$ and $\nabla s_1(x)\in \inter(\cK^\ast)$. 
Similarly to~\eqref{eq:phi-bound}, for $x\succ_\cK y$ we have
\begin{equation*}
\phi(t,x)-\phi(t,y) = e^{\lambda_1 t} v_1 (\nabla s_1(\xi))^T (x-y) + \bar{R}(t)
\end{equation*}
with $x \succ_\cK \xi \succ_\cK y$. Since $\nabla s_1(\xi) \in \inter(\cK^*)$, we have $(x-y)^T \nabla s_1(\xi) >0$ and it follows that $v_1 (\nabla s_1(\xi))^T (x-y) \in \inter(\cK)$. Then \eqref{eq:r-bound} implies
\begin{equation*}
\phi(t,x)-\phi(t,y) \in \inter(\cK) \quad \forall x \succ_\cK y,\,\forall t\geq \tau_0\,,
\end{equation*}
which completes the proof.
\end{proof-of}

\end{document}